\RequirePackage{fix-cm}
\documentclass[smallextended]{svjour3}       
\smartqed  
\usepackage{graphicx}
\usepackage{mathptmx}      
\usepackage{multirow, multicol}
\usepackage{amsmath, amssymb, mathtools} 
\usepackage{stmaryrd, accents}
\usepackage{xcolor}
\usepackage[inline]{enumitem}
\usepackage{tabularx}
\usepackage{empheq} 
\usepackage{tikz}
\usetikzlibrary{calc}
\usepackage{soul}
\usepackage[caption=false]{subfig}

\usepackage[toc,page]{appendix}

\allowdisplaybreaks

\usepackage{hyperref}
\hypersetup{
    colorlinks=true,
    linkcolor=blue,
    filecolor=magenta,
    urlcolor= cyan,
    citecolor=myGreen}
\usepackage{algorithm}
\usepackage{algorithmic}
\usepackage{chngcntr}
\usepackage{accents}
\usepackage{soul}

\usepackage{listings}
\usepackage{mathrsfs}

\usepackage{pifont}
\newcommand{\xmark}{\ding{55}} 
\newcommand{\cmark}{\ding{51}} 

\usepackage{pgfplots}
\pgfplotsset{compat=1.18}  
\usepgfplotslibrary{fillbetween}

\usepackage{caption}
\lstdefinestyle{fancy1}{
  language=Python,
  backgroundcolor=\color{codegray},
  basicstyle=\ttfamily\small,
  keywordstyle=\color{codeblue}\bfseries,
  commentstyle=\color{codegreen}\itshape,
  stringstyle=\color{codered},
  showstringspaces=false,
  frame=tb,
  rulecolor=\color{gray},
  numbers=left,
  numberstyle=\tiny\color{gray},
  breaklines=true,
  captionpos=b,
  tabsize=4
}

\lstdefinestyle{fancy2}{
  language=Python,
  backgroundcolor=\color{codewhite},
  basicstyle=\ttfamily\small,
  keywordstyle=\color{keywordpurple}\bfseries,
  commentstyle=\color{red}\itshape,
  stringstyle=\color{codeblue},
  showstringspaces=false,
  frame=single,
  rulecolor=\color{gray},
  numbers=left,
  numberstyle=\tiny\color{gray},
  breaklines=true,
  captionpos=b,
  tabsize=4
}
\spdefaulttheorem{assumption}{Assumption}{\bf}{\rm}

\newcounter{algostep}[algorithm]

\DeclareMathOperator{\R}{\mathbb{R}}
\DeclareMathOperator{\E}{\mathbb{E}}
\DeclareMathOperator{\X}{\mathcal{X}}
\DeclareMathOperator{\Y}{\mathcal{Y}}
\DeclareMathOperator{\Z}{\mathcal{Z}}
\DeclareMathOperator{\U}{\mathcal{U}}

\DeclareMathOperator{\Lag}{\mathcal{L}}
\DeclareMathOperator{\prox}{\mathrm{Prox}}
\DeclareMathOperator{\proj}{\mathrm{Proj}}
\DeclareMathOperator{\dist}{\mathrm{dist}}

\DeclareMathOperator{\nor}{\mathcal{N}}

\DeclareMathOperator{\Id}{\mathbf{I}}

\DeclareMathOperator{\gp}{\mathrm{gp}}

\DeclareMathOperator{\V}{\mathcal{V}}
\DeclareMathOperator{\F}{\mathcal{F}}

\DeclareMathOperator{\K}{\mathcal{K}}
\DeclareMathOperator{\ceg}{\mathrm{ceg}}
\DeclareMathOperator{\alm}{\mathrm{alm}}
\DeclareMathOperator{\saga}{\mathrm{saga}}

\definecolor{myGreen}{rgb}{0.0, 0.5, 0.0}
\definecolor{codegray}{rgb}{0.95,0.95,0.95}
\definecolor{codeblue}{rgb}{0.0,0.0,0.6}
\definecolor{codegreen}{rgb}{0,0.6,0}
\definecolor{codered}{rgb}{0.6,0,0}
\definecolor{codewhite}{rgb}{0.99, 0.99, 0.99}
\definecolor{keywordpurple}{rgb}{0.3, 0.1, 0.5}

\journalname{Mathematical Programming}

\begin{document}

\title{Primal-Dual Coordinate Descent for Nonconvex-Nonconcave Saddle Point Problems Under the Weak MVI Assumption  
}

\titlerunning{NC-SPDHG for Nonconvex-Nonconcave SPP Under Weak MVI}        

\author{Iyad Walwil         \and
        Olivier Fercoq 
}


\institute{Iyad Walwil \at
              LTCI, Télécom Paris, Institut Polytechnique de Paris, France \\ Department of Mathematics, Faculty of Sciences, An-Najah National University, Nablus, Palestine \\
              \email{iyad.walwil@telecom-paris.fr}  
           \and
           Olivier Fercoq \at
            LTCI, Télécom Paris, Institut Polytechnique de Paris, France \\
            \email{olivier.fercoq@telecom-paris.fr}
}

\date{Received: date / Accepted: date}

\maketitle

\begin{abstract}
We introduce two novel primal-dual algorithms for addressing nonconvex, nonconcave, and nonsmooth saddle point problems characterized by the weak Minty Variational Inequality (MVI). The first algorithm, Nonconvex-Nonconcave Primal-Dual Hybrid Gradient (NC-PDHG), extends the well-known Primal-Dual Hybrid Gradient (PDHG) method to this challenging problem class. The second algorithm, Nonconvex-Nonconcave Stochastic Primal-Dual Hybrid Gradient (NC-SPDHG), incorporates a randomly extrapolated primal-dual coordinate descent approach, extending the Stochastic Primal-Dual Hybrid Gradient (SPDHG) algorithm. 

To our knowledge, designing a coordinate-based algorithm to solve nonconvex-nonconcave saddle point problems is unprecedented, and proving its convergence posed significant difficulties. This challenge motivated us to utilize PEPit, a Python-based tool for computer-assisted worst-case analysis of first-order optimization methods. By integrating PEPit with automated Lyapunov function techniques, we successfully derived the NC-SPDHG algorithm. 

Both methods are effective under a mild condition on the weak MVI parameter, achieving convergence with constant step sizes that adapt to the structure of the problem. Numerical experiments on logistic regression with squared loss and perceptron-regression problems validate our theoretical findings and show their efficiency compared to existing state-of-the-art algorithms, where linear convergence is observed. Additionally, we conduct a convex-concave least-squares experiment to show that NC-SPDHG performs competitively with SAGA, a leading algorithm in the smooth convex setting.
\end{abstract}

\section{Introduction} \label{sec:intro}

Primal-dual algorithms are powerful algorithms for the resolution of complex optimization problems.
By introducing and updating Lagrange multipliers, they are able to split intertwined operators and deal with each of them separately \cite{condat2023proximal}. The counterpart to this simplification is that we need to solve a saddle-point problem, a class of problems that is in general more difficult to solve than minimization problems~\cite{EG+}. 
Although convex-concave saddle point problems have been studied a lot, this is much less true when we relax the convexity assumption. 
The goal of this paper is to design and analyze a primal-dual coordinate-update algorithm under the quite general weak-Minty assumption \cite{weak_MVI}. More precisely, we are going to extend Primal-Dual Hybrid Gradient (PDHG) \cite{PDHG} and Stochastic PDHG (S-PDHG) \cite{SPDHG} so that they can accept gradients and proximal operators of nonconvex functions.
Thus, we are interested in studying the following nonconvex, nonconcave, and nonsmooth saddle point problem:

\begin{equation} \tag{SPP}
    \min_{x \in \X} \max_{y \in \Y} \Lag(x, y) = f(x) + f_2(x) + \langle Ax, y \rangle - g_2(y) - g(y)    
\end{equation}
where functions $f, g$ are (possibly nonconvex) proximal-friendly, functions $f_2, g_2$ are (possibly nonconvex) differentiable with Lipschitz gradients, and $A$ is a linear operator. Our analysis is conducted under the Weak Minty Variational Inequality (MVI) assumption \cite{weak_MVI}, which means that for some $\rho$, and for each $(z, v) \in \mathrm{gph} ~  \partial \Lag$, the inequality, $\langle v, z - z_{\star} \rangle \geq \rho \|v\|^2$, holds, where $z_*$ is a stationary point.   
\subsection{Related work} \label{subsec:related_work}
Many methods have been proposed to address non-convex, non-concave saddle point problems, including extra-gradient-type algorithms \cite{weak_MVI,EG,EG_2,EG+,NP_PDEG}, Primal-Dual Hybrid Gradient (PDHG)-type methods \cite{NL-PDHGM,NP_PDEG,SA-GDmax}, and Augmented Lagrangian-based approaches \cite{PProx-PDA,NC-ADMM,iALM,ALM}. These methods vary in the problem structures they consider and the assumptions they impose.

However, to the best of our knowledge, none of the existing methods provides a primal-dual coordinate descent algorithm for nonconvex-nonconcave saddle point problems. Coordinate descent-based algorithms have been studied for convex-concave saddle point problems \cite{SPDHG,PURE_CD,PDCD_Latafat,PDCD_Fercoq_Bianchi}, as well as for non-convex optimization problems \cite{BPL,NCCD1,NCCD2}, where the focus is solely on the primal objective.

As for PDHG-like methods, only a few variants extend the original PDHG to handle nonconvex-nonconcave saddle point problems. In particular, the work in \cite{NL-PDHGM} considers problems with a nonlinear coupling between the primal and dual variables through a term of the form $\langle K(x), y \rangle$, where $K$ is a nonlinear operator. The authors analyze an extension of PDHG tailored to such settings.

The works in \cite{NP_PDEG,SA-GDmax} are more closely related to ours, as they consider problems of the form $\min_x\max_y f(x) + \phi(x, y) - g(y)$ where $f$ and $g$ are convex, and $\phi$ is a smooth (possibly nonconvex) function. Both works adopt the weak Minty Variational Inequality (MVI) assumption on the sub-differential of the Lagrangian. Specifically, \cite{NP_PDEG} proposes a stochastic, nonlinearly preconditioned primal-dual extra-gradient method, which resembles a PDHG-type algorithm as it updates the dual variable using the most recent primal iterate rather than the previous one.

Meanwhile, the authors in \cite{SA-GDmax} construct a nonlinear variant of PDHG, named the semi-anchored gradient method. It is built upon the theory of the Bregman proximal point method by selecting a Legendre function that directly extends the linear preconditioner of PDHG to a nonlinear one. Their framework proposes exact and inexact versions of their algorithm, with and without projection, and guarantees convergence to an $\varepsilon$-stationary point under the weak MVI assumption.

\subsection{Our contributions} \label{subsec:contributions}
The primary contributions of this paper are two novel primal-dual algorithms designed to solve nonconvex-nonconcave saddle point problems of the form \eqref{eqn:SPP},  characterized by the weak Minty Variational Inequality (MVI) as presented earlier.

\begin{enumerate}
    \item \textbf{Nonconvex-Nonconcave Primal-Dual Hybrid Gradient (NC-PDHG):}
    
     This algorithm generalizes the well-known Chambolle–Pock algorithm \cite{PDHG}, also known as the Primal-Dual Hybrid Gradient (PDHG), which was originally developed for convex–concave saddle point problems. We extend this method to the nonconvex-nonconcave setting and provide a convergence analysis.
    \vspace{0.1cm}
    \item \textbf{Nonconvex-Nonconcave Stochastic Primal-Dual Hybrid Gradient (NC-SPDHG):}
    This algorithm extends the Stochastic Primal-Dual Hybrid Gradient (SPDHG) method \cite{SPDHG} to the nonconvex-nonconcave regime. It incorporates a randomly extrapolated primal-dual coordinate descent approach, making it the first of its kind in this setting. To establish convergence, we leverage the \texttt{PEPit} software \cite{PEPit} along with the automated Lyapunov function technique proposed in \cite{upadhyaya_Auto_Lyapunov}, as formalized in \cite{fercoq2024defininglyapunovfunctionssolution}.
    \vspace{0.1cm}
    \item \textbf{Parameter selection analysis:} For both algorithms, we provide a detailed analysis of parameter choices, showing they function under a mild condition on the weak MVI parameter $\rho$. 
    \vspace{0.1cm}
    \item \textbf{Numerical experiments:} We validate our theoretical findings with experiments on two regression problems:
    (i) logistic regression with squared loss, and (ii) perceptron regression with ReLU activation. Our methods are compared to two state-of-the-art algorithms: Constrained Extra-Gradient+ \cite{EG+}  and the Augmented Lagrangian Method \cite{ALM}. The results confirm the validity of our theoretical analysis and demonstrate the effectiveness of our algorithms, with linear convergence observed empirically.
    \vspace{0.1cm}
    \item \textbf{Convex-concave experiment:} We further test NC-SPDHG on a convex–concave least-squares problem, comparing it with SAGA, a widely used state-of-the-art method. The results highlight the competitive performance of NC-SPDHG even in this classical setting.
\end{enumerate}
\subsection{Organization of the Paper} \label{subsec:organization}
The remainder of the paper is organized as follows:
\begin{enumerate}
\item Section~\ref{sec:clarke} reviews relevant terminology related to Clarke sub-differentials.

\item Section~\ref{sec:assumptions} introduces the saddle point problem of interest along with the key assumptions under which our analysis is conducted.

\item Section~\ref{sec:NC-PDHG} presents our first contribution: the Nonconvex–Nonconcave Primal-Dual Hybrid Gradient (NC-PDHG) algorithm, including its convergence analysis and parameter selection.

\item Section~\ref{sec:SPDHG} outlines our second contribution: the Nonconvex–Nonconcave Stochastic Primal-Dual Hybrid Gradient (NC-SPDHG). We present its convergence result, and analyze its parameter selection.

\item Section~\ref{sec:PEPit} details the computer-assisted proof methodology used for NC-SPDHG, based on the \texttt{PEPit} toolbox and the automated Lyapunov function framework proposed in~\cite{fercoq2024defininglyapunovfunctionssolution}.

\item Section~\ref{sec:SPDHG:convergence} provides the complete convergence proof for the NC-SPDHG algorithm.

\item Section~\ref{sec:numerical_exp} presents numerical experiments that validate our theoretical findings and compare our methods to state-of-the-art algorithms.
\end{enumerate}
\subsection{Notations} \label{subsec:notations}
We denote the primal space by $\X = \prod_{i = 1}^m \X_i$, where $m \geq 1$, and the dual space by $\Y$. The primal-dual space is then  $\Z = \X \times \Y$. The set of saddle points will be denoted as $\Z^{\star} \subset \Z$. All vector spaces are assumed to be Euclidean. Similarly, for a primal vector $x = (x_1, \dots, x_m) \in \X$, and a dual vector $y \in \Y$, we write $z = (x, y) \in \Z$. We define the following norms: 
\begin{itemize}
    \item $\|x\|^2_{\X} = \sum_{i = 1}^m \|x_i\|^2_{\X_i}$.
    \item $\|y\|^2_{\Y}$ denotes the standard $\ell_2$-norm in $\Y$.
    \item $\|z\|^2_{\Z} = \|x\|^2_{\X} + \|y\|^2_{\Y}$.
    \item For $\gamma = (\gamma_x, \gamma_y)$ a couple of positive numbers, $\|z\|^2_{\gamma \Z} = \gamma_x\|x\|^2_{\X} + \gamma_y\|y\|^2_{\Y}$.
\end{itemize}
We consider functions with values in the extended real line $\bar {\mathbb R} := \mathbb R \cup \{+\infty\}$. Given a function $f \colon \X \rightarrow \bar \R$ and a step size $s > 0$, its proximal operator is defined as: 
\begin{equation*}
    \prox_{s, f}(x) = \arg\min_{x' \in \X} f(x') + \frac{1}{2s} \|x' - x\|^2
\end{equation*}
We will make use of the indicator function associated with a subset $S \subset \mathcal{X}$, defined by: 
\begin{equation*}
    \imath_{S}(x) = \begin{cases} 0 & \text{if} ~ x \in S \\ + \infty & \text{Otherwise} \end{cases}
\end{equation*}
We define the projection of a point $u \in \X$ onto a subset $S \subseteq \X$ as  $\displaystyle \proj_{S}(u) = \prox_{\imath_{S}}(u)$. At iteration $k$, the NC-SPDHG algorithm uniformly picks a block index
$i_{k+1} \sim \U\{1, \dots, m\}$. We define as $\mathscr{G}_k$ the
filtration generated by the random indices $\{i_1, \dots, i_k\}$. Lastly, we consider a class of functions $\F$ satisfying some properties like convexity or smoothness. A function $f \in \F$ will be real valued $f \colon \X \rightarrow \R$ and we will consider a list of points $x \in \X$ on which an algorithm $\mathcal{A}$ may act. For a matrix $M$, we will denote by $\|M\|$ its operator norm with respect to the Euclidean norm in its input and output spaces.
\section{Clarke sub-differential} \label{sec:clarke}
As we do not assume convexity, the classical definition of the convex sub-differential no longer applies. Instead, we use the \textit{Clarke sub-differential} \cite{clarke}, which generalizes the classical concept and remains well-defined under weaker conditions.

We start by recalling the definition of \textit{locally Lipschitz} functions — an essential prerequisite for introducing Clarke sub-differentials.
\begin{definition}[Locally-Lipschitz function] \cite[Subsection 2.1]{clarke}

    Let $S \subseteq \X$. A function $f \colon S \rightarrow \R$ satisfies a Lipschitz condition on $S$ if there exists a constant $L \geq 0$ such that:
    \begin{equation} \label{eqn:locally_lipschitz}
        \left|f(x') - f(x) \right| \leq L \left\|x' - x \right\|, ~~~ \forall x', x \in S
    \end{equation}
In particular, 
we say that $f$ is $L$-Lipschitz near $x$ if, for some $\varepsilon > 0, f$ satisfies \eqref{eqn:locally_lipschitz} in the set $x + \varepsilon B$, where $B$ is the open unit ball. 
\end{definition}
We now present a general definition of the \textit{Clarke sub-differential}, which \textit{does not} require the function to be locally Lipschitz, only finite at the point of interest. It is based on the concept of the \textit{Clarke normal cone}.
\begin{definition}[Clarke normal cone] \cite[Proposition 2.5.7]{clarke}
    
    Let $S \subset \X$ be nonempty and closed. The \textit{Clarke normal cone} to $S$ at a point $x \in S$, denoted $\nor_S(x)$, is the closed convex hull: 
    \begin{equation*} \label{eqn:normal_cone}
        \nor_S(x) = \overline{\mathrm{conv}} \left\{ \lambda \lim \frac{v_i}{|v_i|} : \lambda \geq 0, v_i \perp S ~\text{at} ~ x_i, ~ x_i \rightarrow x, ~ v_i \rightarrow 0 \right\}
    \end{equation*}
\end{definition}
where we define a nonzero vector $v$ to be \textit{perpendicular} to $S$ at $x \in S$ (symbolically, $v \perp S ~ \text{at} ~ x$), if there exists $x' \notin S$ whose unique closest point in $S$ is $x$ and such that $v = x' - x$.

\begin{definition}[Clarke sub-differential] \cite[Definition 2.4.10]{clarke}
    
    Let $f \colon \X \rightarrow \R \cup \{\pm \infty\}$ be finite at a point $x \in \X$. The \textit{Clarke sub-differential} of $f$ at $x$, denoted $\partial_C f(x)$, is defined as: 
    \begin{equation*}
        \partial_C f(x) = \left\{\xi \in \X: (\xi, -1) \in \nor_{\mathrm{epi} f}(x, f(x)) \right\}
    \end{equation*}
    where $\nor_{\mathrm{epi} f}$ is the Clarke normal cone to the epigraph of $f$ and the epigraph of $f$, denoted $\mathrm{epi} f$, is defined as: 
    \begin{equation*}
        \mathrm{epi} f = \left\{ (x, t) \in \X \times \R : f(x) \leq t \right\}
    \end{equation*}
\end{definition}
This general definition allows us to recover familiar results from convex analysis for non-Lipschitz functions. One example is the sub-differential of the indicator function of a set, as stated below.
\begin{proposition}\cite[Proposition 2.4.12]{clarke} \label{prop:indicator_clarke}

    Let the point $x \in S$. Then, the Clarke sub-differential of the indicator function $\imath_S$ at $x$ satisfies:
    \begin{equation} \label{eqn:indicator_clarke}
        \partial_C \imath_S(x) = \nor_S(x) 
    \end{equation}
\end{proposition}
We end this section with a simplified version of the \textit{Clarke sub-differential} for locally Lipschitz functions. A well-known result due to Rademacher states that a locally Lipschitz function is differentiable almost everywhere (see e.g., \cite[Theorem 9.60]{rockafellar2009variational}). In particular, every neighborhood of $x$ contains a point $x'$ for which $\nabla f(x')$ exists. This motivates the following simplified construction for the Clarke sub-differential. 
\begin{proposition}\cite[Theorem 2.5.1]{clarke}

Let $f$ be Lipschitz near $x$, and let $\Omega_f$ be the set of points in $x + \varepsilon B$ at which $f$ fails to be differentiable. Then,
    \begin{equation}
        \partial_C f(x) = \mathrm{conv} \left\{ \lim \nabla f(x_i): x_i \rightarrow x, ~ x_i \notin \Omega_f \right\}
    \end{equation}
\end{proposition}

\section{Problem Description and Assumptions} \label{sec:assumptions}
In this paper, we are interested in the following nonconvex, nonconcave, and nonsmooth saddle point problem:

\begin{equation} \tag{SPP} \label{eqn:SPP}
    \min_{x \in \X} \max_{y \in \Y} ~ \Lag(x, y) = f(x) + f_2(x) + \langle Ax, y \rangle - g_2(y) - g(y)    
\end{equation}

under the following assumptions

\begin{assumption}[Problem properties] \label{assump:prob_prop}
    \begin{itemize}
        \item $f \colon \X \rightarrow \bar \R$ and $g \colon \Y \rightarrow \bar \R$ are (possibly nonconvex) proximal-friendly.
        \item $A \colon \X \rightarrow \Y$ is a linear operator.
        \item $f_2 \colon \X \rightarrow \R$ and $g_2 \colon \Y \rightarrow \R$ are (possibly nonconvex) differentiable with $L_{\nabla f_2}$ and $L_{\nabla g_2}$-Lipschitz gradients, respectively. That is, there exist constants $L_{\nabla f_2}, L_{\nabla g_2} \geq 0$ such that:
        \begin{align}
            \|\nabla f_2(x_1) - \nabla f_2(x_2)\| &\leq L_{\nabla f_2} \|x_1 - x_2\|, \hspace{0.5cm} \forall x_1, x_2 \in \X  \label{eqn:f2_Lip}\\ 
            \|\nabla g_2(y_1) - \nabla g_2(y_2)\| &\leq L_{\nabla g_2} \|y_1 - y_2\|, \hspace{0.58cm} \forall y_1, y_2 \in \Y \label{eqn:g2_Lip}
        \end{align}
    \end{itemize}
\end{assumption}
\vspace{0.2cm}
To analyze convergence, we further assume the weak Minty Variational Inequality (MVI) \cite{weak_MVI} condition on the Clarke sub-differential of the Lagrangian:

\begin{assumption}[Weak MVI]\cite[Assumption 1]{weak_MVI} \label{assump:weak-MVI}

There exists a nonempty set $S_{\star} \subseteq \Z_{\star}$ such that for all $z_{\star} \in S_{\star}$ and some $\rho \in \R$ 
    \begin{equation} \label{eqn:weak-MVI}
        \langle v, z - z_{\star} \rangle \geq \rho \|v\|^2, \hspace{1cm} \forall (z, v) \in \mathrm{gph} ~ \partial_C \Lag
    \end{equation}    
\end{assumption}
where $\mathrm{gph} ~\partial_C \Lag$ denotes the graph of the Clarke sub-differential $\partial_C \Lag$. 

\begin{remark}
Note that $\rho$ may be negative.
\end{remark}

\vspace{0.2cm}
\noindent
\textbf{Stationary points.}
Throughout the paper, we assume the existence of a stationary point $z_* = (x_*, y_*)$, which should satisfy the following optimality conditions:
\begin{subequations} \label{eqn:stationary_points}
\begin{align}
    0 \in \partial_x \Lag(x_*, y_*) &= \partial_C f(x_*) + \nabla f_2(x_*) + A^\top y_*  \\ 
    0 \in \partial_y \Lag(x_*, y_*) &= Ax_* - \nabla g_2(y_*) - \partial_C g(y_*)
\end{align}
\end{subequations}
To numerically identify such stationary points, we adopt the definition of the KKT error introduced in \cite{walwil2025smoothed}. Based on the stationarity conditions given in \eqref{eqn:stationary_points}, for any $x, y$, the KKT error is defined as: 
\begin{equation} \label{eqn:KKT_error}
    \K(x, y) = \left\| \partial_C f(x) + \nabla f_2(x) + A^\top y\right\|^2_0  +  \left\| Ax - \nabla g_2(y) - \partial_C g(y)\right\|^2_0
\end{equation}
where we define the "infimal size" of a set $\mathcal{Q}$ as: 
\begin{equation*}
    \|\mathcal{Q}\|_0 = \min \{\|q\| : q \in \mathcal{Q}\}
\end{equation*}
\begin{definition}[$\tau$-stationary point]

Given a tolerance $\tau > 0$, we say that $(x', y')$ is a $\tau$-stationary point if the KKT error defined in \eqref{eqn:KKT_error} satisfies: $\K(x', y') \leq \tau$.    
\end{definition}
\section{NC-PDHG algorithm} \label{sec:NC-PDHG}

The Primal-Dual Hybrid Gradient (PDHG) algorithm is a first-order optimization method designed to solve convex-concave saddle point problems. Originally introduced by Chambolle and Pock in 2011 \cite{PDHG}, PDHG has gained significant popularity in fields such as imaging, signal processing, and machine learning due to its simplicity, scalability, and efficiency for large-scale problems. However, when relaxing the convexity and concavity assumptions, the literature becomes significantly sparser, with only a few recent attempts addressing such problems \cite{NL-PDHGM,NP_PDEG,SA-GDmax}, as previously discussed in the related work subsection \ref{subsec:related_work}.

This section presents our first contribution, in which we propose a Nonconvex-Nonconcave Primal-Dual Hybrid Gradient (NC-PDHG) algorithm that generalizes PDHG to a broader class of saddle point problems of the form \eqref{eqn:SPP}. In particular, NC-PDHG accommodates both gradient and proximal steps involving nonconvex (and nonconcave) functions, enabling the algorithm to address important emerging applications where standard convex assumptions no longer hold. By assuming Assumption \ref{assump:weak-MVI} (weak MVI), and making use of the notion of Clarke sub-differentials presented in Section \ref{sec:clarke}, we managed to tackle this challenging task, and ensure the convergence to zero of an element in the sub-differential of the defined Lagrangian in \eqref{eqn:SPP}.

\begin{algorithm}[H]
\begin{algorithmic}[1]
\STATE \textbf{Input:} $\gamma_x, \gamma_y, \alpha > 0$
\STATE \textbf{Initialize:} $(x_0, y_0)$ 
\WHILE{Stopping Criterion}
    \STATE $\hat{y}_{k+1} = y_k + \gamma_y \left(Ax_k - \nabla g_2(y_k)\right)$ \label{alg_pdhg:yh}
    \STATE $\bar{y}_{k+1} \in \prox_{\gamma_y, g}\left(\hat{y}_{k+1}\right)$ \label{alg_pdhg:yb}
    \STATE $\hat{x}_{k+1} = x_k - \gamma_x \left(\nabla f_2(x_k) + A^\top \bar{y}_{k+1}\right)$ \label{alg_pdhg:xh}
    \STATE $\bar{x}_{k+1} \in \prox_{\gamma_x, f}\left(\hat{x}_{k+1}\right)$ \label{alg_pdhg:xb}
    \STATE $x_{k+1} = x_k + \alpha \left(\bar{x}_{k+1} - \hat{x}_{k+1} - \gamma_x\left(\nabla f_2(\bar{x}_{k+1}) + A^\top \bar{y}_{k+1}\right)\right)$ \label{alg_pdhg:x}
    \STATE $y_{k+1} = y_k + \alpha \left(\bar{y}_{k+1} - \hat{y}_{k+1} + \gamma_y\left(A\bar{x}_{k+1} - \nabla g_2(\bar{y}_{k+1}\right)\right)$ \label{alg_pdhg:y}
\ENDWHILE
\RETURN $(x_{k+1}, y_{k+1})$
\end{algorithmic}
\caption{Nonconvex-Nonconcave Primal-Dual Hybrid Gradient (NC-PDHG)}
\label{alg:NC-PDHG}
\end{algorithm}

Note that most of the steps look very similar to the formulation of PDHG in \cite{PDCD_Latafat}.
The main differences come from the non-uniqueness of the points in the proximal operator, the extrapolation in $\nabla f_2$ and $\nabla g_2$ and the presence of the parameter $\alpha \leq 1$.

\subsection{Convergence Proof} \label{sec:PDHG:convergence}
Before proving the main theorem of this section,  we show that we can have an easy access to a point in $\partial_C \mathcal L(\bar z_{k+1})$.
\begin{lemma} \label{lemma:D_k+1}
Let $\bar D_{k+1} = (\bar F_{k+1}, \bar G_{k+1})$ be defined as follows:
    \begin{align}
    \begin{split} \label{eqn:G_k+1}
        \bar G_{k+1} &= \frac{1}{\gamma_y} (\hat y_{k+1} - \bar y_{k+1}) + \nabla g_2(\bar y_{k+1}) - A\bar x_{k+1} \\
        &= \frac{1}{\gamma_y}(y_k - \bar y_{k+1}) + \nabla g_2(\bar y_{k+1}) - \nabla g_2(y_k)  +  A(x_k - \bar x_{k+1}) 
    \end{split} \\
    \begin{split} \label{eqn:F_k+1}
        \bar F_{k+1} &= \frac{1}{\gamma_x}(\hat x_{k+1} - \bar x_{k+1}) + \nabla f_2(\bar x_{k+1}) + A^\top \bar y_{k+1} \\
         &= \frac{1}{\gamma_x} (x_k - \bar x_{k+1}) + \nabla f_2(\bar x_{k+1}) - \nabla f_2(x_k) 
    \end{split}
    \end{align}
    Then, for $(\bar x_{k+1}, \bar y_{k+1})$ generated by Algorithm \ref{alg:NC-PDHG}, we obtain:
    \begin{align}
        \bar G_{k+1} &\in \partial_C g(\bar y_{k+1}) + \nabla g_2(\bar y_{k+1}) - A \bar x_{k+1} \label{eqn:G_k+1_weak_MVI}\\ 
        \bar F_{k+1} &\in \partial_C f(\bar x_{k+1}) + \nabla f_2(\bar x_{k+1}) + A^\top \bar y_{k+1} \label{eqn:F_k+1_weak_MVI}
    \end{align}
\end{lemma}
\begin{proof}
    We prove the first inclusion, and the second follows similarly. According to Algorithm \ref{alg:NC-PDHG}, $\bar y_{k+1}$ is updated as: 
    \begin{align*}
        \bar{y}_{k+1} &\in \prox_{\gamma_y, g}(\hat y_{k+1}) = \arg\min_{y' \in \Y} g(y') + \frac{1}{2\gamma_y} \|y' - \hat y_{k+1}\|^2_{\Y} \\ 
        &\Longrightarrow 0 \in \partial \left(g + \frac{1}{2\gamma_y}\|. - \hat y_{k+1}\|^2_{\Y} \right)(\bar y_{k+1}) \\
        & \Longrightarrow 0 \in \partial_C g(\bar y_{k+1}) + \frac{1}{\gamma_y} (\bar y_{k+1} - \hat y_{k+1}) \\ 
        & \Longrightarrow \frac{1}{\gamma_y} (y_k - \bar y_{k+1}) + Ax_k - \nabla g_2(y_k) \in \partial_C g(\bar y_{k+1}) \\
        &\Longrightarrow \bar G_{k+1} \in \partial_C g(\bar y_{k+1}) + \nabla g_2(\bar y_{k+1}) - A\bar x_{k+1} ~~~~~~~~~~ \qed 
    \end{align*}
\end{proof}

\begin{theorem}
    Assume that Assumptions \ref{assump:prob_prop} and \ref{assump:weak-MVI} hold. Let $\underaccent{\bar}{\gamma} = \min\left(\gamma_x, \gamma_y\right), C = \alpha \left(  \frac{\underaccent{\bar}{\gamma}(1 - \alpha)}{2} + \rho\right)$, the extrapolation parameter, $\alpha$, and the step sizes, $\gamma_x$ and $\gamma_y$, satisfy: 
    \begin{align} \label{eqn:step sizes, NC-PDHG}
        \alpha < 1 + \frac{2\rho}{\underaccent{\bar}{\gamma}}, && \gamma_y \leq \frac{1}{\sqrt{2}L_{\nabla g_2}}, && 
        2\gamma_x \gamma_y \|A\|^2 + \gamma_x^2 L^2_{\nabla f_2} \leq 1 
    \end{align}
    Then, for $z_k = (x_k, y_k)$ generated by Algorithm \ref{alg:NC-PDHG}, it follows that: 
    \begin{equation}
        \sum_{k = 0}^{K-1} \left\|\bar D_{k+1}\right\|^2_{\Z} \leq \frac{\|z_0 - z_*\|^2_{\gamma^{-1} \Z}}{2C}
    \end{equation}
\end{theorem}
\begin{proof} We consider the following Lyapunov function 
\begin{equation*}
    \V(z_{k+1}, z_*) = \frac{1}{2} \|z_{k+1} - z_*\|^2_{\gamma^{-1} \Z} := \frac{1}{2\gamma_x} \|x_{k+1} - x_*\|^2_{\X} + \frac{1}{2\gamma_y} \|y_{k+1} - y_*\|^2_{\Y}
\end{equation*}
\begin{enumerate}
    \item We start by applying Assumption \ref{assump:weak-MVI} on our problem. By Lemma \ref{lemma:D_k+1}, since
\begin{align*}
\bar F_{k+1} &\in \partial_C f(\bar x_{k+1}) + \nabla f_2 (\bar x_{k+1}) + A^\top \bar y_{k+1} \\
\bar G_{k+1} &\in \partial_C g(\bar y_{k+1}) + \nabla g_2 (\bar y_{k+1}) - A \bar x_{k+1}
\end{align*}
then: 
\begin{equation}\label{eqn:PDHG_weak_MVI}
    \left\langle \bar F_{k+1}, \bar x_{k+1} - x_* \right\rangle_{\X} + \left\langle \bar G_{k+1}, \bar y_{k+1} - y_* \right\rangle_{\Y} \geq \rho \left(\left\|\bar F_{k+1}\right\|^2_{\X} + \left\|\bar G_{k+1} \right\|^2_{\Y} \right) 
\end{equation}
    \item We develop the primal quantity, $\|x_{k+1} - x_*\|^2_{\X}$.
\begin{align}
    \frac{1}{2\gamma_x} &\|x_{k+1} - x_*\|^2_{\X} \notag \\
    &\stackrel{\text{Step \ref{alg_pdhg:x}}}{=} \frac{1}{2\gamma_x} \left\|x_k - x_* + \alpha \left( \bar{x}_{k+1} - \hat{x}_{k+1} - \gamma_x\left(\nabla f_2(\bar{x}_{k+1}) + A^\top \bar{y}_{k+1}\right) \right)\right\|^2_{\X} \notag\\ 
    &= \frac{1}{2\gamma_x} \|x_k - x_*\|^2_{\X} + \frac{\alpha^2}{2\gamma_x}\left\|\bar{x}_{k+1} - \hat{x}_{k+1} - \gamma_x \left(\nabla f_2(\bar{x}_{k+1}) + A^\top \bar{y}_{k+1}\right)\right\|^2_{\X} \notag\\ 
    & \qquad\qquad\qquad\qquad + \frac{\alpha}{\gamma_x} \left\langle x_k - x_*, \bar{x}_{k+1} - \hat{x}_{k+1} - \gamma_x \left(\nabla f_2(\bar{x}_{k+1}) + A^\top \bar{y}_{k+1}\right) \right\rangle_{\X} \notag\\ 
    & \stackrel{(\ref{eqn:F_k+1})}{=}  \frac{1}{2\gamma_x} \|x_k - x_*\|^2_{\X} + \frac{\alpha^2 }{2\gamma_x}\left\| -\gamma_x \bar F_{k+1} \right\|^2_{\X} + \frac{\alpha}{\gamma_x} \left\langle x_k - x_*, -\gamma_x \bar F_{k+1}\right\rangle_{\X} \notag\\ 
    & = \frac{1}{2\gamma_x} \|x_k - x_*\|^2_{\X} + \frac{\alpha^2 \gamma_x}{2}\left\| \bar F_{k+1}\right\|^2_{\X} - \alpha \left\langle x_k - \bar x_{k+1}, \bar F_{k+1} \right\rangle_{\X} - \alpha \left\langle \bar x_{k+1} - x_*, \bar F_{k+1} \right\rangle_{\X}
    \label{eq:xx_ncpdhg}
\end{align}
Now, we further study the following scalar term: 
\begin{align*}
    -\alpha & \left\langle x_k - \bar{x}_{k+1}, \bar F_{k+1} \right\rangle_{\X} \\
    &= - \frac{\alpha \gamma_x}{2} \left\|\frac{1}{\gamma_x} (x_k - \bar x_{k+1}) \right\|^2_{\X} - \frac{\alpha \gamma_x}{2} \left\| \bar F_{k+1} \right\|^2_{\X} + \frac{\alpha \gamma_x}{2} \left\| \frac{1}{\gamma_x} (x_k - \bar x_{k+1}) - \bar F_{k+1} \right\|^2_{\X} \\
    &\stackrel{\eqref{eqn:F_k+1}}{=} - \frac{\alpha \gamma_x}{2} \left\|\frac{1}{\gamma_x} (x_k - \bar x_{k+1}) \right\|^2_{\X} - \frac{\alpha \gamma_x}{2} \left\| \bar F_{k+1} \right\|^2_{\X} + \frac{\alpha \gamma_x}{2} \left\| \nabla f_2(x_k) - \nabla f_2(\bar{x}_{k+1}) \right\|^2_{\X} \\
    &\stackrel{\eqref{eqn:f2_Lip}}{\leq} - \frac{\alpha \gamma_x}{2} \left\|\frac{1}{\gamma_x} (x_k - \bar x_{k+1}) \right\|^2_{\X} - \frac{\alpha \gamma_x}{2} \left\| \bar F_{k+1} \right\|^2_{\X} + \frac{\alpha \gamma_x L_{\nabla f_2}^2}{2} \left\| x_k - \bar{x}_{k+1} \right\|^2_{\X} \\
    & = - \frac{\alpha \gamma_x}{2} \left\|\frac{1}{\gamma_x} (x_k - \bar x_{k+1}) \right\|^2_{\X} - \frac{\alpha \gamma_x}{2} \left\| \bar F_{k+1} \right\|^2_{\X} + \frac{\alpha \gamma_x^3 L_{\nabla f_2}^2}{2} \left\| \frac{1}{\gamma_x} (x_k - \bar{x}_{k+1}) \right\|^2_{\X}  \\
    & = \frac{\alpha \gamma_x}{2} \left(\gamma_x^2 L_{\nabla f_2}^2 - 1\right) \left\|\frac{1}{\gamma_x} (x_k - \bar x_{k+1}) \right\|^2_{\X} - \frac{\alpha \gamma_x}{2} \left\| \bar F_{k+1} \right\|^2_{\X}
\end{align*}
Hence, substituting what we got in \eqref{eq:xx_ncpdhg}, we obtain: 
\begin{align*}
    \frac{1}{2\gamma_x} & \|x_{k+1} - x_*\|^2_{\X} \leq \frac{1}{2\gamma_x} \|x_k - x_*\|^2_{\X} + \frac{\alpha \gamma_x}{2} (\alpha - 1)\left\| \bar F_{k+1}\right\|^2_{\X} - \alpha \left\langle \bar x_{k+1} - x_*, \bar F_{k+1} \right\rangle_{\X} \\ 
    &\qquad\qquad\qquad\qquad\qquad\qquad\qquad + \frac{\alpha \gamma_x}{2} \left(\gamma_x^2 L_{\nabla f_2}^2 - 1\right) \left\|\frac{1}{\gamma_x} (x_k - \bar x_{k+1}) \right\|^2_{\X}
\end{align*}
\item Similarly, we develop the dual quantity, $\|y_{k+1} - y_*\|^2_{\Y}$.
\begin{align*}
   & \frac{1}{2\gamma_y}  \|y_{k+1} - y_*\|^2_{\Y} \\
    &\stackrel{\text{Step }\ref{alg_pdhg:y}, \eqref{eqn:G_k+1}}{=}  \frac{1}{2\gamma_y} \|y_k - y_*\|^2_{\Y} + \frac{\alpha^2 \gamma_y}{2}\left\| \bar G_{k+1}\right\|^2_{\Y} - \alpha \left\langle y_k - \bar y_{k+1}, \bar G_{k+1} \right\rangle_{\Y} - \alpha \left\langle \bar y_{k+1} - y_*, \bar G_{k+1} \right\rangle_{\Y}
\end{align*}
Then, we further study the term $ - \alpha \left\langle y_k - \bar{y}_{k+1}, \bar G_{k+1} \right\rangle_{\Y} $:
\begin{align*}
    - \alpha & \left\langle y_k - \bar{y}_{k+1}, \bar G_{k+1} \right\rangle_{\Y} \\
    &= -\frac{\alpha \gamma_y}{2} \left\|\frac{1}{\gamma_y} (y_k - \bar y_{k+1})\right\|^2_{\Y} - \frac{\alpha \gamma_y}{2} \|\bar G_{k+1}\|^2_{\Y} + \frac{\alpha \gamma_y}{2} \left\|\frac{1}{\gamma_y} (y_k - \bar y_{k+1}) - \bar G_{k+1} \right\|^2_{\Y} \\
    &= -\frac{\alpha \gamma_y}{2} \left\|\frac{1}{\gamma_y} (y_k - \bar y_{k+1})\right\|^2_{\Y} - \frac{\alpha \gamma_y}{2} \|\bar G_{k+1}\|^2_{\Y} \\ 
    & \qquad \hspace{3.5cm} \qquad+ \frac{\alpha \gamma_y}{2} \left\| \nabla g_2(y_k) - \nabla g_2(\bar y_{k+1}) + A (\bar x_{k+1} - x_k)\right\|^2_{\Y} \\
    &\leq -\frac{\alpha \gamma_y}{2} \left\|\frac{1}{\gamma_y} (y_k - \bar y_{k+1})\right\|^2_{\Y} - \frac{\alpha \gamma_y}{2} \|\bar G_{k+1}\|^2_{\Y} + \alpha \gamma_y \left\| \nabla g_2(y_k) - \nabla g_2(\bar y_{k+1})\right\|^2_{\Y} \\
    &\qquad \hspace{7cm} \qquad+ \alpha \gamma_y \left\|A (\bar x_{k+1} - x_k)\right\|^2_{\Y} \\
    &\stackrel{\eqref{eqn:g2_Lip}}{\leq} -\frac{\alpha \gamma_y}{2} \left\|\frac{1}{\gamma_y} (y_k - \bar y_{k+1})\right\|^2_{\Y} - \frac{\alpha \gamma_y}{2} \|\bar G_{k+1}\|^2_{\Y} + \alpha \gamma_y L_{\nabla g_2}^2 \left\|y_k - \bar y_{k+1}\right\|^2_{\Y} \\
    & \qquad \hspace{7cm} \qquad + \alpha \gamma_y \|A\|^2 \left\|\bar x_{k+1} - x_k\right\|^2_{\X} \\
    &= \alpha \gamma_y \left(\gamma_y^2 L_{\nabla g2}^2 - 0.5\right) \left\|\frac{1}{\gamma_y} (y_k - \bar y_{k+1})\right\|^2_{\Y} - \frac{\alpha \gamma_y}{2} \|\bar G_{k+1}\|^2_{\Y} + \alpha \gamma_y \|A\|^2 \left\|\bar x_{k+1} - x_k\right\|^2_{\X}
\end{align*}
Substituting, we obtain 
\begin{align*}
    \frac{1}{2\gamma_y} & \|y_{k+1} - y_*\|^2_{\Y} 
    \leq \frac{1}{2\gamma_y} \|y_k - y_*\|^2_{\Y} + \frac{\alpha \gamma_y}{2} (\alpha - 1)\left\| \bar G_{k+1}\right\|^2_{\Y} - \alpha \left\langle \bar y_{k+1} - y_*, \bar G_{k+1} \right\rangle_{\Y} \\
    &\qquad \qquad \qquad + \alpha \gamma_y \left(\gamma_y^2 L_{\nabla g2}^2 - 0.5\right) \left\|\frac{1}{\gamma_y} (y_k - \bar y_{k+1})\right\|^2_{\Y} + \alpha \gamma_y \|A\|^2 \left\|\bar x_{k+1} - x_k\right\|^2_{\X}
\end{align*}
\item Lastly, we combine everything to measure the primal-dual quantity, $\|z_{k+1} - z_*\|^2_{\gamma^{-1} \Z}$:
\begin{align*}
    \frac{1}{2} & \|z_{k+1} - z_*\|^2_{\gamma^{-1} \Z} = \frac{1}{2\gamma_x}\|x_{k+1} - x_*\|^2_{\X} + \frac{1}{2\gamma_y}\|y_{k+1} - y_*\|^2_{\Y} \\
    &\leq \frac{1}{2\gamma_x} \|x_k - x_*\|^2_{\X} + \frac{1}{2\gamma_y} \|y_k - y_*\|^2_{\Y} - \frac{\alpha \gamma_x}{2} (1 - \alpha)\left\| \bar F_{k+1}\right\|^2_{\X} - \frac{\alpha \gamma_y}{2} (1 - \alpha)\left\| \bar G_{k+1}\right\|^2_{\Y}  \\
    & \qquad - \alpha \left\langle \bar x_{k+1} - x_*, \bar F_{k+1} \right\rangle_{\X} - \alpha \left\langle \bar y_{k+1} - y_*, \bar G_{k+1} \right\rangle_{\Y} + \alpha \gamma_y \|A\|^2 \left\|\bar x_{k+1} - x_k\right\|^2_{\X} \\
    &\qquad  + \frac{\alpha \gamma_x}{2} \left(\gamma_x^2 L_{\nabla f_2}^2 - 1\right) \left\|\frac{1}{\gamma_x} (x_k - \bar x_{k+1}) \right\|^2_{\X} + \frac{\alpha \gamma_y}{2} \left(2\gamma_y^2 L_{\nabla g2}^2 - 1\right) \left\|\frac{1}{\gamma_y} (y_k - \bar y_{k+1})\right\|^2_{\Y} \\
    &\leq \frac{1}{2\gamma_x} \|x_k - x_*\|^2_{\X} + \frac{1}{2\gamma_y} \|y_k - y_*\|^2_{\Y} - \alpha \underaccent{\bar}{\gamma} \frac{(1 - \alpha)}{2} \left\| \bar F_{k+1}\right\|^2_{\X} - \alpha \underaccent{\bar}{\gamma} \frac{(1 - \alpha)}{2} \left\| \bar G_{k+1}\right\|^2_{\Y} \\
    & \qquad - \alpha \left\langle \bar x_{k+1} - x_*, \bar F_{k+1} \right\rangle_{\X} - \alpha \left\langle \bar y_{k+1} - y_*, \bar G_{k+1} \right\rangle_{\Y} + \alpha \gamma_y \|A\|^2 \left\|\bar x_{k+1} - x_k\right\|^2_{\X} \\
    &\qquad  + \frac{\alpha \gamma_x}{2} \left(\gamma_x^2 L_{\nabla f_2}^2 - 1\right) \left\|\frac{1}{\gamma_x} (x_k - \bar x_{k+1}) \right\|^2_{\X} + \frac{\alpha \gamma_y}{2} \left(2\gamma_y^2 L_{\nabla g2}^2 - 1\right) \left\|\frac{1}{\gamma_y} (y_k - \bar y_{k+1})\right\|^2_{\Y} \\
    &\leq \frac{1}{2} \|z_k - z_*\|^2_{\gamma^{-1}\Z} - \alpha \underaccent{\bar}{\gamma} \frac{(1 - \alpha)}{2} \left\|\bar D_{k+1}\right\|^2_{\Z} - \alpha \left\langle \bar z_{k+1} - z_*, \bar D_{k+1} \right\rangle_{\Z} \\
    &\qquad  \hspace{2.2cm} \qquad+ \frac{\alpha \gamma_x}{2} \left(2 \gamma_x \gamma_y \|A\|^2 + \gamma_x^2 L_{\nabla f_2}^2 - 1\right) \left\|\frac{1}{\gamma_x} (x_k - \bar x_{k+1}) \right\|^2_{\X} \\ 
    & \qquad \hspace{4.5cm} \qquad+ \frac{\alpha \gamma_y}{2} \left(2\gamma_y^2 L_{\nabla g2}^2 - 1\right) \left\|\frac{1}{\gamma_y} (y_k - \bar y_{k+1})\right\|^2_{\Y} \\
    &\leq \frac{1}{2} \|z_k - z_*\|^2_{\gamma^{-1}\Z} - \alpha \left(\frac{\underaccent{\bar}{\gamma}(1 - \alpha)}{2} + \rho\right) \left\|\bar D_{k+1}\right\|^2_{\Z} + \frac{\alpha \gamma_y}{2} \left(2\gamma_y^2 L_{\nabla g2}^2 - 1\right) \left\|\frac{1}{\gamma_y} (y_k - \bar y_{k+1})\right\|^2_{\Y}  \\
    &\qquad  \hspace{2.2cm} \qquad+ \frac{\alpha \gamma_x}{2} \left(2 \gamma_x \gamma_y \|A\|^2 + \gamma_x^2 L_{\nabla f_2}^2 - 1\right) \left\|\frac{1}{\gamma_x} (x_k - \bar x_{k+1}) \right\|^2_{\X} 
\end{align*}
\end{enumerate}
Thus, taking the parameters according to the conditions in \eqref{eqn:step sizes, NC-PDHG}
\begin{align*}
    \frac{\underaccent{\bar}{\gamma}(1 - \alpha)}{2} + \rho > 0, &&
    2\gamma_y^2 L_{\nabla g2}^2 - 1 < 0, &&
    2 \gamma_x \gamma_y \|A\|^2 + \gamma_x^2 L_{\nabla f_2}^2 - 1 < 0
\end{align*}
implies:
\begin{equation*}
\|\bar D_{k+1}\|^2_{\Z} \leq \frac{1}{2C} \left(\|z_k - z_*\|^2_{\gamma^{-1}\Z} - \|z_{k+1} - z_*\|^2_{\gamma^{-1}\Z}\right)
\end{equation*}
Taking the sum from $k = 0$ to $k = K-1$, and telescoping implies the result. \qed

\end{proof}

\begin{remark}
    By choosing to measure distances in $\Z$ with step-size-dependent norms, it would have been possible to replace $ \|\bar D_{k+1}\|_{\Z}$ by $\|\bar D_{k+1}\|_{\gamma\Z}$ \cite{QEB_SDG}. We did not follow this approach because it implies a dependence of the weak-MVI constant on the step sizes. 
\end{remark}

\subsection{Parameters analysis} \label{subsec:parameters_NC_PDHG}
In this subsection,  we analyze the choices for the step sizes $\gamma_x, \gamma_y$, and the extrapolation parameter $\alpha$ presented in \eqref{eqn:step sizes, NC-PDHG}, in the simplified setting where $f_2 = 0$ and assuming $\rho \leq 0$. The conditions then reduce to:
\begin{align*}
    \alpha < 1 + \frac{2 \rho}{\underaccent{\bar}{\gamma}} && \gamma_y \leq \frac{1}{\sqrt{2}L_{\nabla g_2}} && 2 \gamma_x \gamma_y \|A\|^2 \leq 1
\end{align*}
Rewriting the first condition as $\underaccent{\bar}{\gamma} > \frac{2|\rho|}{1 - \alpha}$, we get the following two main conditions: 
\begin{align*}
    \frac{2|\rho|}{1 - \alpha} < \gamma_y \leq \frac{1}{\sqrt{2}L_{\nabla g_2}} && \frac{2|\rho|}{1 - \alpha} < \gamma_x \leq \frac{1}{2 \gamma_y \|A\|^2}
\end{align*}
To examine whether feasible step sizes exist, we begin with the case $\alpha = 0$, yielding:
\begin{align*}
    2|\rho| < \gamma_y \leq \frac{1}{\sqrt{2}L_{\nabla g_2}} && 2|\rho| < \gamma_x \leq \frac{1}{2 \gamma_y \|A\|^2}
\end{align*}
Let us select $$\gamma_y = 2|\rho| + \varepsilon$$ 
for some $\varepsilon > 0$ satisfying 
$$ 0 < \varepsilon \leq \frac{1}{\sqrt{2}L_{\nabla g_2}}  - 2|\rho|$$
We now substitute this into the upper bound of $\gamma_x$, leading to the condition: $$ 2|\rho| < \gamma_x \leq \frac{1}{(4 |\rho| + 2\varepsilon) \|A\|^2} $$
This inequality has a feasible solution if the right-hand side exceeds the lower bound, that is:
$$\frac{1}{(4 |\rho| + 2\varepsilon) \|A\|^2} - 2|\rho| > 0$$ 
This holds if $$\varepsilon < \frac{1 - 8\rho^2 \|A\|^2}{4 |\rho| \|A\|^2}$$

Lastly, To improve the convergence speed of PDHG-like algorithms, one may consider residual balancing techniques as discussed in \cite{residual_balancing}, which typically require careful tuning of algorithm parameters. In our context, this could be implemented experimentally by scaling the original perturbation $\varepsilon$ to a smaller value by multiplying it by a user-defined constant $c$ for some $c \in (0, 1)$.

Since the existence of feasible parameters is ensured in the base case $\alpha = 0$, we now return to the general case and propose a concrete choice of parameters:
\begin{enumerate}
    \item We check if $\rho^2 < \max \left\{ \frac{1}{8\|A\|^2}, \frac{1}{8L_{\nabla g_2}^2} \right\}$ and proceed if it's true. Otherwise, the problem is too complex to be solved by NC-PDHG (Algorithm \ref{alg:NC-PDHG}).
    \item For some $c \in (0, 1)$, choose 
    $$\varepsilon = c \times \min \left( \frac{1}{\|A\|}, \frac{1 - 8\|A\|^2 \rho^2}{4 \|A\|^2 |\rho|}, \frac{1}{\sqrt{2}L_{\nabla g_2}}  - 2 |\rho| \right)$$ where the first term is inspired by the convex-concave setting (i.e., $\rho \to 0$), and the other two ensure feasibility of $\gamma_x$ and $\gamma_y$.
    \item Define $\gamma_x, \gamma_y$, and $\alpha$ as:
    \begin{align*}
        \gamma_y = 2|\rho| + \varepsilon &&
        \gamma_x = \frac{1}{2 \gamma_y \|A\|^2} &&
        \alpha = 1 + \frac{2\rho}{\underaccent{\bar}{\gamma}}
    \end{align*}
\end{enumerate}
\begin{remark}
    Parameters selection has been set in a way that adapts to both nonconvex-nonconcave and convex-concave problems, meaning that if your saddle point problem is originally convex-convex, our conditions will still be valid. 
\end{remark}

\section{NC-SPDHG algorithm} \label{sec:SPDHG}
In this section, we present the second and main contribution of this paper: a nonconvex–nonconcave coordinate-based primal-dual algorithm. As discussed in the related work subsection \ref{subsec:related_work}, coordinate descent algorithms have been extensively studied in two distinct contexts: convex–concave saddle-point problems \cite{SPDHG,PURE_CD,PDCD_Latafat,PDCD_Fercoq_Bianchi}, and nonconvex optimization problems \cite{BPL,NCCD1,NCCD2}, where the focus lies solely on the primal variable.

To the best of our knowledge, this work constitutes the first proposal of a coordinate-based primal-dual algorithm designed for nonconvex–nonconcave saddle-point problems.

Coordinate-based algorithms typically require a separability assumption: the functions associated with the variables updated in a block-wise or coordinate-wise fashion must be separable. Therefore, in this section, we assume that the objective function $f$ is separable and that $f_2 = 0$. Under these assumptions, problem \eqref{eqn:SPP} simplifies to:

\begin{equation} \label{eqn:SPP_SPDHG} \tag{CD-SPP}
\min_{x \in \X} \max_{y \in \Y} \sum_{i = 1}^m f_i(x_i) + \langle A_i x, y \rangle - g_2(y) - g(y),
\end{equation}
where the primal variable $x$ is split into $m$ blocks, and the dual variable $y$ is updated globally.

We now introduce our Nonconvex–Nonconcave Stochastic Primal-Dual Hybrid Gradient (NC-SPDHG) algorithm, which solves \eqref{eqn:SPP_SPDHG} under Assumption \ref{assum:prob_prop_CD} (problem structure) and Assumption \ref{assump:weak-MVI} (weak Minty variational inequality condition). The proposed algorithm extends the SPDHG algorithm \cite{SPDHG}, originally developed for convex–concave saddle-point problems, to the nonconvex–nonconcave setting.

\begin{assumption}(Problem properties) \label{assum:prob_prop_CD}
    \begin{itemize}
        \item The primal domain is decomposed as $\X = \prod_{i = 1}^m \X_i$ with $\dim(\X) = d$, and each block $\X_i$ satisfies $\dim(\X_i) = d_i$ for $i = 1, \dots, m$, where $\sum_{i = 1}^m d_i = d$.
        \item $f \colon \X \rightarrow \bar \R$ is (possibly nonconvex) proximal-friendly, and block-separable, i.e., $\displaystyle f(x) = \sum_{i = 1}^m f_i(x_i), ~ x_i \in \X_i$.
        \item $A_i \colon \X_i \rightarrow \Y$ is a linear operator, for all $i \in \{1, \dots, m\}$
        \item $g \colon \Y \rightarrow \bar \R$ is (possibly nonconvex) proximal-friendly.
        \item $g_2 \colon \Y \rightarrow \R$ is (possibly nonconvex) differentiable with $L_{\nabla g_2}$-Lipschitz gradient.
    \end{itemize}
\end{assumption}

\begin{algorithm}[H] 
\begin{algorithmic}[1]
\STATE \textbf{Input:} $\gamma_x, \gamma_y, \alpha, \theta > 0$
\STATE \textbf{Initialize:} $(x_0, y_0)$
\WHILE{Stopping Criterion}
    \STATE $\hat{y}_{k+1} = y_k + \gamma_y \left(Ax_k - \nabla g_2(y_k)\right)$
    \STATE $\bar{y}_{k+1} \in \prox_{\gamma_y, g}\left(\hat{y}_{k+1}\right)$
    \STATE $\hat{x}_{k+1} = x_k - \gamma_x A^\top \bar{y}_{k+1}$
    \STATE $\bar{x}_{k+1} \in \prox_{\gamma_x, f}\left(\hat{x}_{k+1}\right)$
    \STATE Draw a block index $i_{k+1} \sim \U\{1, \dots, m\}$
    \STATE $\displaystyle x_{k+1}^j = \begin{cases} (1 - \alpha) x_k^j + \alpha \bar{x}_{k+1}^j, & \text{if} ~~ j \in \X_{i_{k+1}} \\ x_k^j, & \text{otherwise} \end{cases}$
    \STATE $y_{k+1} = (1 - \alpha ) y_k + \alpha \bar{y}_{k+1} + \gamma_y \theta A(x_{k+1} - x_k) + \alpha \gamma_y \left(\nabla g_2(y_k) - \nabla g_2(\bar{y}_{k+1})\right)$ 
\ENDWHILE
\RETURN $(x_{k+1}, y_{k+1})$
\end{algorithmic}
\caption{Nonconvex-Nonconcave Stochastic Primal-Dual Hybrid Gradient (NC-SPDHG)}
\label{alg:NC-SPDHG}
\end{algorithm}

The algorithm exhibits two key features. The first is its \textit{block coordinate-based structure}, wherein only a single block of the primal variable is updated at each iteration. This design significantly reduces the per-iteration computational cost, making the algorithm particularly suitable for large-scale problems. Although this approach may require a larger number of iterations to achieve convergence compared to full-coordinate methods, the overall computational burden is often lower. Intuitively, this approach resembles solving a sequence of lower-dimensional subproblems instead of tackling a single high-dimensional one—akin to solving $n$ one-dimensional optimization problems rather than a single $n$-dimensional problem.
The second feature is the \textit{stochastic dual extrapolation}, where the dual variable is updated not only via a proximal step but also adjusted using a correction based on the randomly updated coordinates of the primal variable. This is achieved through the term $A(x_{k+1} - x_k)$, which captures the effect of updating a single coordinate in the primal iterate. As a result, the dual extrapolation incorporates stochastic information from the primal side, enabling efficient coordinate-wise updates. 

\begin{remark}
A key distinction between the dual updates in Algorithm \ref{alg:NC-PDHG} and Algorithm \ref{alg:NC-SPDHG}, in the case where $f_2 = 0$, lies in the extrapolation terms. Specifically, the dual update in NC-PDHG involves the term $\gamma_y A(\bar x_{k+1} - x_k)$, while in NC-SPDHG it involves $\gamma_y \theta A(x_{k+1} - x_k)$. Importantly, the latter is significantly cheaper to compute, as it only requires evaluating the change in a single block (due to the block update rule), whereas the former requires computing over the full extrapolated vector $\bar x_{k+1}$. This makes the NC-SPDHG algorithm more computationally efficient in practice, particularly in large-scale settings.
\end{remark}

The convergence of the algorithm is provided in the next theorem: 

\begin{theorem} \label{thm:NC-SPDHG}
    Assume that Assumption \ref{assump:weak-MVI} and \ref{assum:prob_prop_CD} hold. Let the step sizes, $\gamma_x$ and $\gamma_y$, and the extrapolation parameter, $\alpha$, satisfy: 
\begin{subequations} \label{eqn:SPDHG_param_cond}
    \begin{align}
        \gamma_y &\leq \frac{1}{\sqrt{2} L_{\nabla g_2}} \label{eqn:gammay_g2} \\
         \alpha &< 1 + \frac{2\rho}{\gamma_y} \label{eqn:gammay_alpha} \\
        0 & <\rho + \gamma_x \left(1 - \frac{\alpha}{2}\right) + \gamma_x^2 \gamma_y \left(\left(1 - \frac{\alpha}{2}\right)\|A\|^2 + \frac{m\alpha}{2} \sup_j \|Ae_j\|^2\right) \label{eqn:gammax_gammay}
    \end{align}
\end{subequations}
    Then, for $z_k = (x_k, y_k)$ generated by Algorithm \ref{alg:NC-SPDHG}, it follows that:
\begin{equation}
    \sum_{k = 0}^{K-1} \E\left[\left\|\bar D_{k+1}\right\|^2_{\Z}\right] \leq \frac{1}{2\alpha C}\left(\frac{m}{\gamma_x} \|x_{0} - x_*\|^2_{\X} + \frac{1}{\gamma_y} \|y_{0} - y_*\|^2_{\Y}\right)
\end{equation}
where $C = \min(C_x, C_y), C_x = \rho + \gamma_x \left(1 - \frac{\alpha}{2}\right) - \gamma_x^2 \gamma_y \left(\left(1-\frac{\alpha}{2}\right)\|A\|^2 + \alpha m \sup_\ell \frac{\|A_\ell\|^2}{2}\right)$, $C_y = \rho + \frac{\gamma_y}{2} (1-\alpha)$ and $\bar D_{k+1} \in \partial_C\mathcal L(\bar z_{k+1})$ is defined in Lemma~\ref{lemma:D_k+1}.
\end{theorem}
\begin{proof}
    See Section \ref{sec:SPDHG:convergence}
\end{proof}

\subsection{Parameters analysis} \label{subsec:parameters_NC_SPDHG}
In this subsection,  we analyze the choices of the step sizes $\gamma_x, \gamma_y$, and the extrapolation parameter $\alpha$ presented in \eqref{eqn:SPDHG_param_cond}. 

From conditions \eqref{eqn:gammay_g2} and \eqref{eqn:gammay_alpha}, we obtain the following feasibility range for $\gamma_y$:
\begin{equation} \label{eqn:gammay_bounds}
    \frac{2|\rho|}{1 - \alpha} < \gamma_y \leq \frac{1}{\sqrt{2} L_{\nabla g_2}}
\end{equation}
To check the feasibility of these conditions, we consider the simplified case $\alpha = 0$. Substituting this into \eqref{eqn:gammax_gammay} and \eqref{eqn:gammay_bounds}, we obtain:
\begin{align*}
   |\rho| - \gamma_x + \gamma_y\gamma_x^2\|A\|^2 < 0 && 2|\rho| < \gamma_y \leq \frac{1}{\sqrt{2} L_{\nabla g_2}}
\end{align*}
Hence, we select $\gamma_y = 2|\rho| + \varepsilon$ for some $\varepsilon > 0$ satisfying $ 0 < \varepsilon \leq \frac{1}{\sqrt{2}L_{\nabla g_2}}  - 2|\rho|$.
Substituting this into the first inequality, we search for $\gamma_x$ satisfying: 
$$ |\rho| - \gamma_x + (\varepsilon + 2|\rho|)\|A\|^2 \gamma_x^2 < 0$$ 
This inequality admits a solution if the discriminant $\Delta$ of the corresponding quadratic equation is positive:
\begin{align*}
    \Delta := 1 - 4|\rho| \left(\varepsilon + 2|\rho|\right)\|A\|^2 > 0 \iff \varepsilon < \frac{1 - 8\|A\|^2 \rho^2}{4 \|A\|^2 |\rho|}
\end{align*}
If this condition holds, then the set of valid values for $\gamma_x$ is given by:
\begin{equation*}
    \frac{1 - \sqrt{\Delta}}{(2\varepsilon + 4|\rho|)\|A\|^2} < \gamma_x < \frac{1 + \sqrt{\Delta}}{(2\varepsilon + 4|\rho|)\|A\|^2}
\end{equation*}
Having shown that a feasible set of step sizes exists when $\alpha = 0$, we now turn back to the full conditions and provide a concrete parameter selection procedure:
\begin{enumerate}
    \item We check if $\rho^2 < \max \left\{ \frac{1}{8\|A\|^2}, \frac{1}{8L_{\nabla g_2}^2} \right\}$ and proceed if it's true. Otherwise, the problem is too complex to be solved by NC-SPDHG.
    \item For some $c \in (0, 1)$, choose 
    $$\varepsilon = c \times \min \left( \frac{1}{\|A\|}, \frac{1 - 8\|A\|^2 \rho^2}{4 \|A\|^2 |\rho|}, \frac{1}{\sqrt{2}L_{\nabla g_2}}  + 2\rho \right)$$ 
    This ensures that $\gamma_y$ remains feasible, the discriminant $\Delta$ is positive, and the choice remains consistent with what is known from the convex case (i.e., $\rho \to 0$). Moreover, the parameter $c$ serves as a tuning knob that influences the recalibration of the step sizes to achieve better residual balancing, as discussed in subsection \ref{subsec:parameters_NC_PDHG} and in line with the approach in \cite{residual_balancing}.
    \item Define $\gamma_x, \gamma_y$, and $\alpha$ as: 
    \begin{align*}
        \gamma_y &= 2|\rho| + \varepsilon \\ 
        \gamma_x &= \frac{1}{2\gamma_y\|A\|^2} \\ 
        \alpha &= \min \left( 1 - \frac{2|\rho|}{\gamma_y}, \frac{2(\gamma_x - \gamma_x^2 \gamma_y \|A\|^2 - |\rho|)}{\gamma_x + \gamma_x^2 \gamma_y (m \sup_j \|A_j\|^2 - \|A\|^2)}\right)
    \end{align*}
\end{enumerate}
\begin{remark}
    Parameters selection has been set in a way that adapts to both nonconvex-nonconcave and convex-concave problems, meaning that if your saddle point problem is originally convex-convex, our conditions will still be valid. 
\end{remark}
\section{PEPiting NC-SPDHG} \label{sec:PEPit}
In this section, we recount the development process of our NC-SPDHG algorithm, guided by recent advances in convergence analysis techniques introduced in \cite{fercoq2024defininglyapunovfunctionssolution}. Specifically, we leverage the \texttt{PEPit} \cite{PEPit} performance estimation toolbox in combination with the automated Lyapunov function search methodology proposed in \cite{upadhyaya_Auto_Lyapunov}. We begin by outlining the motivation behind adopting these tools in our analysis.

\subsection{Introduction}
To analyze the convergence of complex optimization algorithms, a widely adopted technique is to construct a Lyapunov function — a nonnegative function whose value strictly decreases across iterations — along with a residual function that captures a quantifiable decrease in the objective. Mathematically, this entails identifying a Lyapunov function $\mathcal{V} \geq 0$ such that:
\begin{equation} \label{eqn:Lyapunov_inequality}
\mathcal{V}(f, \mathcal{A}(f, x)) + \mathcal{R}(f, x) \leq \mathcal{V}(f, x), \quad \forall f \in \mathcal{F}, ~ \forall x \in \mathcal{X},
\end{equation}
where $\mathcal{A}(f, x)$ is the outcome of one iteration of algorithm $\mathcal{A}$ to a function $f$ at point $x$, and $\mathcal{R}$ is a residual function.

Traditionally, such Lyapunov inequalities are derived manually by combining algorithmic update rules with known structural properties of the objective. However, this process is highly problem-dependent and often requires considerable analytical effort.

In our work, we address the challenge of designing a coordinate-based primal-dual algorithm for nonconvex–nonconcave saddle point problems, a setting where convergence guarantees are scarce. While previous methods such as PURE-CD \cite{PURE_CD} and SPDHG \cite{SPDHG} have been developed for convex–concave problems, and algorithms like NP-PDEG \cite{NP_PDEG} target nonconvex–nonconcave regimes, adapting or combining their analysis techniques to get an algorithm that works proved insufficient for our setting. This difficulty led us to adopt a recent methodology introduced in \cite{upadhyaya_Auto_Lyapunov,fercoq2024defininglyapunovfunctionssolution}, which reformulates the search for a Lyapunov function as a solution to a performance estimation saddle point problem. This reformulation allows one to automatically verify the existence of a suitable Lyapunov function and hence numerically validate convergence guarantees.

This line of work builds on the theory of Performance Estimation Problems (PEPs) \cite{taylorthesis,PEPit1,PEPit2}, which is implemented in the \texttt{PEPit} Python library \cite{PEPit}. Given a first-order method and a problem class, PEPit automatically reformulates worst-case performance estimation as a semi-definite program (SDP).

Further developments by \cite{upadhyaya_Auto_Lyapunov} demonstrated that PEPs can also be leveraged to automatically detect the existence of quadratic Lyapunov functions, thereby enabling automated convergence analysis. For instance, their framework significantly broadened the feasibility range of step sizes for which convergence of the Chambolle–Pock algorithm \cite{larger_range_chambolle} could be certified.

Building upon this, the search for quadratic Lyapunov functions was reinterpreted in \cite{fercoq2024defininglyapunovfunctionssolution}  as a convex-concave saddle point problem within the PEP framework. This problem can be reformulated using PEPit and solved using the DSP-CVXPY solver \cite{DSP-CVXPY}.  Formally, the problem reduces to:
\begin{equation} \label{eqn:quad_Lyapuov}
    \min_{\{\V \in \mathcal{Q}\}} \max_{\{f \in \F, x \in \X\}} \V(f, \mathcal{A}(f, x)) + \mathcal{R}(f, x) - \V(f, x) \leq 0 
\end{equation}
where $\mathcal{Q}$ is a parametric class of quadratic Lyapunov functions. The inner maximization problem can be formulated as a performance estimation problem whose constraints (semidefinite, linear inequalities, and equalities) are automatically generated by \texttt{PEPit}, enabling the overall expression to be cast as a bilinear convex-concave saddle point problem involving semi-definite variable, which is solvable using DSP-CVXPY \cite{DSP-CVXPY}.  We refer the reader to~\cite{fercoq2024defininglyapunovfunctionssolution} for more details about the reformulation and its implementation within the PEP framework.

\subsection{Negative results in \texttt{PEPit}}
Trying to combine various analysis techniques from both the convex-concave and nonconvex-nonconcave frameworks proved insufficient for our setting, as previously discussed. Yet, our journey toward designing a convergent coordinate-based algorithm using \texttt{PEPit} was far from straightforward.

In this subsection, we share one of our unsuccessful attempts to construct such an algorithm for solving the saddle point problem \eqref{eqn:SPP_SPDHG}. Interestingly, this version failed to converge even in the convex setting, where both $f$ and $g$ are convex. 

\begin{algorithm}[H] 
\begin{algorithmic}[1]
\STATE \textbf{Input:} $\gamma_x, \gamma_y, \alpha, \theta > 0$
\STATE \textbf{Initialize:} $(x_0, y_0)$
\WHILE{Stopping Criterion}
    \STATE $\hat{y}_{k+1} = y_k + \gamma_y \left(Ax_k - \nabla g_2(y_k)\right)$
    \STATE $\bar{y}_{k+1} = \prox_{\gamma_y, g}\left(\hat{y}_{k+1}\right)$
    \STATE $\hat{x}_{k+1} = x_k - \gamma_x A^\top \bar{y}_{k+1}$
    \STATE $\bar{x}_{k+1} = \prox_{\gamma_x, f}\left(\hat{x}_{k+1}\right)$
    \STATE Draw $i_{k+1} \sim \U\{1, \dots, m\}$
    \STATE $\displaystyle x_{k+1}^j = \begin{cases} (1 - \alpha) x_k^j + \alpha \bar{x}_{k+1}^j, & \text{if} ~~ j = i_{k+1} \\ x_k^j, & \text{otherwise} \end{cases}$
    \STATE $y_{k+1} = (1 - \alpha ) y_k + \alpha \bar{y}_{k+1} + \alpha \gamma_y \theta A(x_{k+1} - x_k) + \alpha \gamma_y \left(\nabla g_2(y_k) - \nabla g_2(\bar{y}_{k+1})\right)$ 
\ENDWHILE
\RETURN $(x_{k+1}, y_{k+1})$
\end{algorithmic}
\caption{Failed algorithm}
\label{alg:failed}
\end{algorithm}
Using the methodology discussed earlier, we attempted to analyze this algorithm with \texttt{PEPit}, but were unable to find a Lyapunov function satisfying condition \eqref{eqn:Lyapunov_inequality}. Upon deeper investigation, we identified the critical issue: the presence of the term $\alpha \gamma_y \theta A (x_{k+1} - x_k)$ in the $y_{k+1}$-update. 

Intuitively, we want the updates of NC-SPDHG to mimic those of the deterministic NC-PDHG algorithm \textit{in expectation}. In particular, this means that the NC-SPDHG update for $y_{k+1}$ should follow the same structure on average as the corresponding NC-PDHG update. As a result, the term $\alpha \gamma_y \theta A(x_{k+1} - x_k)$ should not be scaled by $\alpha$, contrary to what we initially implemented.

Indeed, as shown in Lemma \ref{lem:NC-SPDHG_expec}, taking the expectation of this term reveals that the correct form is $\gamma_y \theta A(x_{k+1} - x_k)$, without the additional $\alpha$ factor. Incorporating this correction was a key insight that led to the successful design of the convergent NC-SPDHG algorithm presented in Section \ref{sec:SPDHG}.

\subsection{Coding randomized algorithms in PEPit}
Coding (randomized) algorithms to solve optimization problems for numerical illustrations differs significantly from implementing them in \texttt{PEPit} for the purpose of analyzing convergence. This distinction becomes even trickier when the algorithm includes randomization.

In this subsection, we first outline the general differences when coding \textit{any} algorithm. We then highlight the additional considerations specific to \textit{randomized} algorithms, using the failed algorithm (Algorithm~\ref{alg:failed}) as an illustrative example.

\vspace{1em}
\textbf{General Difference (for any algorithm):}
In practical implementations, we typically run an algorithm over many iterations to observe convergence behavior. However, in \texttt{PEPit}, the goal is not to simulate long-term behavior but rather to perform worst-case analysis via Lyapunov functions. Specifically, \texttt{PEPit} attempts to certify convergence by searching for a valid Lyapunov function that decreases between two successive iterates.

Thus, instead of running the algorithm for many steps, we define a single iteration: we define the starting point, say $(x_0, y_0)$, then apply one iteration of the algorithm to obtain $(x_1, y_1)$. Note \texttt{PEPit} looks for the worst case behavior, optimizing over the problem class as well as the initialization $(x_0, y_0)$. Hence, if a suitable Lyapunov function is found to decrease from $(x_0, y_0)$ to $(x_1, y_1)$, it implies that the same function is valid for all subsequent iterates under the same update rule.

\vspace{1em}
\textbf{Additional Difference (for randomized algorithms):}
In the case of randomized algorithms, defining only a single iteration introduces a subtle issue: drawing a random variable just once has no statistical significance. Therefore, instead of sampling a random variable, \texttt{PEPit} requires that we define the behavior of the algorithm across all possible outcomes of the randomness. That is, we must explicitly define how the algorithm updates its state for each possible realization of the random variable.

To make this distinction concrete, we consider the failed algorithm (Algorithm~\ref{alg:failed}) and present two implementations:

\begin{itemize}
    \item A standard Python implementation used for numerical experiments, where randomness is realized by sampling. 
    \item A \texttt{PEPit}-style implementation, in which all possible outcomes of the randomized update rule are explicitly defined to enable worst-case analysis.
\end{itemize}
We begin with the standard implementation used for numerical testing.

\begin{lstlisting}[style=fancy1, 
caption={\textbf{Listing 1} Code of Algorithm \ref{alg:NC-SPDHG} (NC-SPDHG) for numerical illustrations},
label={lst:practice}
]
import numpy as np

# Define the proximal of f and g, with s being the corresponding step size
def prox_f(x, s):
    ...
    
def prox_g(y, s):
    ...

# Define the gradient of g2 
def grad_g2(y):
    ...

def failed_algorithm(xo, yo, prox_f, prox_g, grad_g2, gammax, gammay, alpha, theta, stopping_criterion):
    x, y = xo, yo
    m = len(x)
    
    while stopping_criterion:
        grad_g2_y = grad_g2(y)
        # Dual forward step
        y_hat = y + gammay * (A @ x - grad_g2_y)
        # Dual backward step
        y_bar = prox_g(y_hat, gammay)

        # Draw a random variable
        i = np.random.randint(0, m)

        # Primal-coordinate forward step
        x_hat_i = x[i] - gammax * A[:, i] @ y_bar
        # Primal-coordinate forward step
        x_bar_i = prox_f(x_hat_i, gammax, i)
        
        # Primal extrapolation 
        x_new_i = (1 - alpha) * x[i] + alpha * x_bar_i
        # Dual extrapolation 
        y = (1 - alpha) * y + alpha * y_bar \
            + alpha * theta * gammay * A[i] * (x_new_i - x[i]) \
            + alpha * gammay * (grad_g2_y - grad_g2(y_bar))
            
        x[i] = x_new_i
        
    return x, y
\end{lstlisting}
We now present, step by step, how to implement the failed algorithm (Algorithm \ref{alg:failed}) within the \texttt{PEPit} framework. We begin by including the appropriate \texttt{Python} imports:

\begin{lstlisting}[style=fancy2]
import numpy as np
from PEPit import PEP
from PEPit.functions import ConvexFunction, SmoothFunction
from PEPit.operators import LinearOperator
from PEPit.primitive_steps import proximal_step
\end{lstlisting}

To set up our problem, we start by specifying the parameter values. Note the following points:
\begin{itemize}
    \item We select $m = 2$, representing the simplest case for the saddle point problem \eqref{eqn:SPP_SPDHG}. If the algorithm converges in this basic scenario, it is a promising indication that it may also converge for any $m > 2$.
    \item The convergence or divergence of the algorithm depends on the particular choice of parameters. That is, with different values of some or all of the parameters we may have different results.
\end{itemize}
\begin{lstlisting}[style=fancy2]
# Define the paramters, for instance:
m = 2
rho = -0.07
gammax = 0.5
gammay = 0.5
alpha = 0.75
theta = m
\end{lstlisting}
Then, we initialize a PEP object, which allows manipulating the forthcoming ingredients of the PEP, such as functions and iterates. 
\begin{lstlisting}[style=fancy2]
problem = PEP()

# To maintain separability, we define the following lists of lists
f, M, A, AT = [[]] * m, [[]] * m, [[]] * m, [[]] * m

for i in range(m):
    f[i] = problem.declare_function(ConvexFunction)
    M[i] = problem.declare_function(LinearOperator, L = 1)
    A[i] = M[i].gradient
    AT[i] = M[i].T.gradient
g2 = problem.declare_function(SmoothFunction, L = 1)
g = problem.declare_function(ConvexFunction)
\end{lstlisting}
To interpret different spaces within \texttt{PEPit}, we need to partition its problem space accordingly. For example, when using \texttt{PEPit} for saddle point problems solved by deterministic algorithms such as PDHG, we partition the space into two subspaces representing the primal and dual spaces. However, for coordinate-based algorithms, the primal space itself must be further sub-partitioned — in our case, into $m$ primal blocks, resulting in $m$-primal blocks plus one dual block. That is:
\begin{lstlisting}[style=fancy2]
partition = problem.declare_block_partition(d = 1+m)
\end{lstlisting}
Next, we define the initial point, with particular attention to correctly extracting the primal and dual blocks. 
\begin{lstlisting}[style=fancy2]
z = problem.set_initial_point('z')
x = [partition.get_block(z, j) for j in range(m)]
y = partition.get_block(z, m)
\end{lstlisting}
Now that everything is set, we proceed to define one iteration of the failed algorithm in \texttt{PEPit}, which we break down into three steps:
\begin{enumerate}
    \item The dual forward-backward step.
\begin{lstlisting}[style=fancy2,
caption={\textbf{Listing 2a} Dual forward-backward step in Algorithm \ref{alg:failed} according to \texttt{PEPit}}]
Ax_dev = [[]] * m
for i in range(m):
    Ax_dev[i] = A[i](x[i])
Ax = np.sum([Ax_dev[i] for i in range(m)])

grad_g2_y = g2.gradient(y)

# Dual forward step
y_hat = y + gammay * (Ax - grad_g2_y)

# Dual backward step
y_bar, _, _ = prox(y_hat, g, gammay)
\end{lstlisting}
    \item The primal forward-backward step.
\begin{lstlisting}[style=fancy2,
caption={\textbf{Listing 2b} Primal forward-backward step in Algorithm \ref{alg:failed} according to \texttt{PEPit}}]
ATy_bar = [[]] * m
for i in range(m):
    ATy_bar[i] = AT[i](y_bar)
    
x_hat, x_bar, diff_x_bar = [[]] * m, [[]] * m, [[]] * m
for i in range(m):
    # Primal forward step
    x_hat[i] = x[i] - gammax * ATy_bar[i]

    # Primal backward step
    x_bar[i], _, _ = prox(x_hat[i], f[i], gammax)
\end{lstlisting}
    \item Random extrapolation steps: We draw a random variable $l \sim \U\{0, \dots, m-1\}$ and explicitly enumerate all possible outcomes of the algorithm corresponding to each value of $l$.
\begin{lstlisting}[style=fancy2,
caption={\textbf{Listing 2c} Primal-Dual extrapolation steps in Algorithm \ref{alg:failed} according to \texttt{PEPit}}]
x1, y1 = [[]] * m, [[]] * m
grad_g2_y_bar = g2.gradient(y_bar)

for l in range(m):
    x1[l] = [[]] * m

    # Primal extrapolation
    for i in range(m):
        if i == l:
            x1[l][i] = (1 - alpha) * x[i] + alpha * xb1[i]
        else:
            x1[l][i] = x[i]

    # Dual extrapolation
    Ax1_l = np.sum([A[i](x1[l][i]) for i in range(m)])
    y1[l] = (1 - alpha) * y + alpha * y_bar + \
            alpha * gammay * theta * (Ax1_l - Ax) + \ 
            alpha * gammay * (grad_g2_y - grad_g2_y_bar)
\end{lstlisting}
\end{enumerate}

As discussed earlier, removing the factor $\alpha$ from the term $\alpha \gamma_y \theta A(x_{k+1} - x_k)$in the update of $y_{k+1}$ leads to a convergent algorithm. This adjustment formed the basis of our initial version of Algorithm \ref{alg:NC-SPDHG} (NC-SPDHG), which assumed partial convexity by requiring $f$ and $g$ to be convex. Subsequently, we succeeded in relaxing this convexity assumption, allowing the algorithm to handle fully nonconvex-nonconcave Lagrangian. We provide a detailed explanation of this extension in the following subsection.
    
\subsection{Nonconvex proximal operators}
Before arriving at the final version of Algorithm \ref{alg:NC-SPDHG} (NC-SPDHG), we initially assumed that both $f$ and $g$ were convex. We refer to this preliminary case as \textit{version 1} of Algorithm \ref{alg:NC-SPDHG}. The purpose of this subsection is to present the approach we followed to relax the convexity assumption in version 1 while still ensuring the convergence of the algorithm.

We get inspiration by the approach introduced in \cite{gribonval2020characterization}, where the authors characterize proximity operators of (possibly nonconvex) penalties as sub-differentials of some convex potentials, as stated in Theorem \ref{thm:nonconvex_prox} presented below. Extending Moreau's result \cite[Proposition 1]{gribonval2020characterization} to possibly nonconvex functions. 

\begin{theorem} \cite[Theorem 1]{gribonval2020characterization} \label{thm:nonconvex_prox}
    Let $\mathcal{H}$ be a Hilbert space, let $\mathcal{U} \subset \mathcal{H}$. A function $\prox_{\varphi} \colon \mathcal{U} \rightarrow \mathcal{H}$ is a proximity operator of a function $\varphi \colon \mathcal{H} \rightarrow \bar \R$ if, and only if, there exists a convex lower semi-continuous function $\psi \colon \mathcal{H} \rightarrow \bar \R$ such that for each $u \in \mathcal{U}, \prox_{\varphi}(u) \in \partial \psi(u)$.
\end{theorem}
By integrating this characterization into \texttt{PEPit}, we move from working with purely nonconvex functions or proximal operators, which often lack useful analytical properties, to working with convex potentials. Specifically, instead of expressing the updates as $\bar x_{k+1} \in \prox_{\gamma_x f}(\hat x_{k+1})$ and $\bar y_{k+1} \in \prox_{\gamma_y g}(\hat y_{k+1})$, we now write them as $\bar x_{k+1} \in \partial \psi_f(\hat x_{k+1})$ and $\bar y_{k+1} \in \partial \psi_g(\hat y_{k+1})$ for some convex l.s.c functions $\psi_f$ and $\psi_g$. This allows \texttt{PEPit} to search for the worst-case convex functions $\psi_f$ and $\psi_g$ such that a Lyapunov function $\mathcal{V}$ exists. If such a Lyapunov function is identified, it will be valid for any convex potentials $\psi_f$ and $\psi_g$. Consequently, the identified Lyapunov function is also valid for the proximal operator of any functions $\varphi_f$ and $\varphi_g$.

Next, we highlight the main change in the \texttt{PEPit} code when relaxing the convexity assumptions on $f$ and $g$, and incorporating the characterization given in Theorem \ref{thm:nonconvex_prox}.
\begin{enumerate}
\item We define convex potentials $\psi_f$ and $\psi_g$. This piece of code remains almost identical to the previous version, except that  $\psi_f$ and $\psi_g$ are now used to implicitly represent nonconvex functions.
\begin{lstlisting}[style=fancy2]
psi_f, M, A, AT = [[]] * m, [[]] * m, [[]] * m, [[]] * m

for i in range(m):
    psi_f[i] = problem.declare_function(ConvexFunction)
    M[i] = problem.declare_function(LinearOperator, L = 1)
    A[i] = M[i].gradient
    AT[i] = M[i].T.gradient
g2 = problem.declare_function(SmoothFunction, L = 1)
psi_g = problem.declare_function(ConvexFunction)
\end{lstlisting}
    \item The primal and dual backward updates write as follows now: 
\begin{lstlisting}[style=fancy2]
# Primal-Dual forward step 
y_hat = ...     # Same as before
x_hat = ...     # Same as before

# Primal-Dual backward step
y_bar = psi_g.gradient(y_hat) 
x_bar = psi_f.gradient(x_hat)
\end{lstlisting}
    and the remainder of the code remains unchanged.
\end{enumerate}
At this point, we had our algorithm fully implemented within \texttt{PEPit}, and we successfully identified a Lyapunov function that guarantees convergence—i.e., one for which the inequality \eqref{eqn:Lyapunov_inequality} holds. This marked the first major step in transitioning from a mere \textit{conjecture} to a \textit{theorem} supported by an analytical proof.

Leveraging \texttt{PEPit}’s structure for defining problems and variables, we extracted corresponding numerical values. These values guided us in formulating a candidate Lyapunov function and identifying a set of key inequalities to consider—such as the weak (MVI) at a specific point, the smoothness of $g_2$, among others.

This process ultimately enabled us to rigorously prove our conjecture, now formalized in Theorem \ref{thm:NC-SPDHG}. The complete proof is provided in the following section.
\section{Convergence proof of NC-SPDHG} \label{sec:SPDHG:convergence}
\begin{lemma} \label{lem:NC-SPDHG_expec}
For $\bar x_{k+1}$ and $x_{k+1}$ generated by Algorithm \ref{alg:NC-SPDHG}, and for $\bar G_{k+1}$ defined in \eqref{eqn:G_k+1}, let $\theta = m$ and define $H_{k+1} := \bar G_{k+1} + A(\bar x_{k+1} - x_k) - \frac{\theta}{\alpha} A(x_{k+1} - x_k)$. Then, the following holds:  
    \begin{align}
        \E_k\left[\theta A(x_{k+1} - x_k)\right] &= \alpha A(\bar x_{k+1} - x_k)  \label{eqn:E_RE_PDHG}\\
        \E_k[H_{k+1}] &= \bar G_{k+1} \label{eqn:E_H_k+1}\\
        \E_k\left[\left\langle H_{k+1}, y_k - y_* \right \rangle_{\Y} \right] &= \left \langle \bar G_{k+1}, y_k - y_* \right \rangle_{\Y} \label{eqn:H_k+1_inner} \\ 
        \E_k \left[ \left\| H_{k+1} \right\|^2_{\Y} \right] &= \left\|\bar G_{k+1}\right\|^2_{\Y} + \frac{\theta^2}{\alpha^2} \E_k\left[\left\|A( x_{k+1} - x_k)\right\|^2_{\Y} \right] - \left\|A(\bar x_{k+1} - x_k)\right\|^2_{\Y} \label{eqn:H_k+1_norm} \\
        \frac{\theta^2}{\alpha^2} \E_k\left[\|A(x_{k+1} - x_k)\|^2_{\Y}\right] &\leq  m \sup_{1 \leq \ell \leq m} {\|A_\ell\|^2 \|\bar x_{k+1} - x_k\|^2_{\X}} \label{eqn:E_x_tilde_norm}  
    \end{align}
\end{lemma}
\begin{proof}
    \begin{enumerate}
        \item Recall that $x_{k+1}^j = x_k^j, ~ \forall j \neq b_{i_{k+1}}$, then: 
    \begin{align*}
        \E_k\left[\theta A(x_{k+1} - x_k)\right] &= \theta \E_k\left[ \sum_{i = 1}^m A_i \left(x_{k+1}^i - x_k^i \right)\right] \\ 
        &= \theta \E_k\left[A_{i_{k+1}} \left(x_{k+1}^{i_{k+1}} - x_k^{i_{k+1}}\right)\right] \\ 
        &= \theta \E_k\left[\alpha A_{i_{k+1}} \left(\bar x_{k+1}^{i_{k+1}} - x_k^{i_{k+1}}\right)\right] \\ 
        &= \frac{\alpha \theta}{m} \sum_{\ell = 1}^m A_\ell \left(\bar x_{k+1}^{\ell} - x_k^\ell \right) \\
        &= \alpha A \left(\bar x_{k+1} - x_k \right) 
    \end{align*}
        \item By the fact that $\bar G_{k+1} \in \mathcal{G}_k$, we have: 
        \begin{equation*}
            \E_k[H_{k+1}] = \bar G_{k+1} + A(\bar x_{k+1} - x_k) -\frac{\theta}{\alpha} \E_k \left[A(x_{k+1} - x_k)\right] \stackrel{\eqref{eqn:E_RE_PDHG}}{=} \bar G_{k+1}
        \end{equation*}
        \item Eq. \eqref{eqn:H_k+1_inner} follows directly from \eqref{eqn:E_H_k+1}.
        \item We use the identity $\E_k\left[\|\xi\|^2\right] = \E_k\left[\left\| \xi - \E_{k}[\xi]\right\|^2\right] + \left\|\E_k[\xi]\right\|^2$ with $\xi = H_{k+1}$.
        \begin{align*}
            \E_k\left[\|H_{k+1}\|^2_{\Y}\right] 
            &\stackrel{\eqref{eqn:E_H_k+1}}{=} \E_k\left[\left\|H_{k+1} - \bar G_{k+1}\right\|^2_{\Y}\right] + \left\|\bar G_{k+1}\right\|^2_{\Y} \\ 
            &= \E_k\left[\left\|A(\bar x_{k+1} - x_k) - \frac{\theta}{\alpha} A(x_{k+1} - x_k)\right\|^2_{\Y}\right] + \left\|\bar G_{k+1}\right\|^2_{\Y}
        \end{align*}
        We reuse the same identity with $\xi = \frac{\theta}{\alpha} A(x_{k+1} - x_k)$ to get: 
        \begin{align*}
             \E_k\left[\left\|A(\bar x_{k+1} - x_k) - \frac{\theta}{\alpha} A( x_{k+1} - x_k)\right\|^2_{\Y}\right] \\ 
             & \!\hspace{-2cm}\! = \frac{\theta^2}{\alpha^2} \E_k\left[\left\|A(x_{k+1} - x_k)\right\|^2_{\Y}\right] - \left\|A(\bar x_{k+1} - x_k)\right\|^2_{\Y}
        \end{align*}
        \item 
        \begin{align*}
            \frac{\theta^2}{\alpha^2} \E_k\left[\|A(x_{k+1} - x_k)\|^2_{\Y}\right] &= \frac{m^2}{\alpha^2} \E_k\left[\left\|A_{i_{k+1}} \left(x_{k+1}^{i_{k+1}} - x_k^{i_{k+1}}\right)\right\|^2_{\Y}\right] \\
            &= \frac{m^2}{\alpha^2} \E_k\left[\left\|A_{i_{k+1}} \left( (1 - \alpha) x_k^{i_{k+1}} + \alpha \bar x_{k+1}^{i_{k+1}} - x_k^{i_{k+1}} \right)\right\|^2_{\Y} \right] \\
            &= m^2 \E_k\left[\left\|A_{i_{k+1}} \left(\bar x_{k+1}^{i_{k+1}} - x_k^{i_{k+1}}\right)\right\|^2_{\Y} \right] \\
            &= m \sum_{\ell = 1}^m \left\|A_\ell \left(\bar x_{k+1}^\ell - x_k^\ell\right)\right\|^2_{\Y} \\
            & \leq m \sup_{1 \leq \ell\leq m} \|A_\ell\|^2 \sum_{\ell = 1}^m \|\bar x_{k+1}^\ell - x_k^\ell\|^2_{\X} \\
            &= m \sup_{1 \leq \ell \leq m} \|A_\ell\|^2 \|\bar x_{k+1} - x_k \|^2_{\X}
        \end{align*}
        \qed
    \end{enumerate} 
\end{proof}

\begin{itemize}[label={$\blacktriangleright$}]
    \item \textbf{Proof of Theorem} \ref{thm:NC-SPDHG}
\end{itemize}
\begin{proof}
We consider the following Lyapunov function 
\begin{equation*}
    \V(z_{k+1}, z_*) = \E_k\left[\frac{m}{2\gamma_x} \|x_{k+1} - x_*\|^2_{\X} + \frac{1}{2\gamma_y} \|y_{k+1} - y_*\|^2_{\Y} \right]
\end{equation*}
\begin{enumerate}
    \item We start by applying Assumption \ref{assump:weak-MVI} on our problem. Since Algorithm \ref{alg:NC-SPDHG} updates $(\bar x_{k+1}, \bar y_{k+1})$ in the same way as Algorithm \ref{alg:NC-PDHG}, then we can use Lemma \ref{lemma:D_k+1}, since 
\begin{align*}
    \bar F_{k+1} = \frac{1}{\gamma_x} (x_k - \bar x_{k+1}) &\in \partial_c f(\bar x_{k+1}) + A^\top \bar y_{k+1} \\
    \bar G_{k+1} &\in \partial_c g(\bar y_{k+1}) + \nabla g_2(\bar y_{k+1}) - A\bar x_{k+1}
\end{align*}
Thus, for some $\rho \neq 0$, we have:
\begin{equation}
    \left\langle \bar x_{k+1} - x_*, \bar F_{k+1} \right\rangle_{\X} + \left\langle \bar y_{k+1} - y_*, \bar G_{k+1} \right\rangle_{\Y} \geq \rho \left(\|\bar F_{k+1}\|^2_{\X} + \|\bar G_{k+1}\|^2_{\Y} \right) 
\end{equation}
\item We study the primal quantity, $\E_k\left[\|x_{k+1} - x_*\|^2_{\X} \right]$. Recall that $x_{k+1}^j = x_k^j$ if $j \neq i_{k+1}$. Hence,
\begin{align*}
\frac{1}{2\gamma_x} & \E_k \left[\|x_{k+1} - x_*\|^2_{\X} - \|x_k - x_*\|^2_{\X} \right] \\
&= \frac{1}{2\gamma_x} \E_k\left[\sum_{i = 1}^m \left\|x_{k+1}^i - x_*^i\right\|^2_{\X_i} - \sum_{i = 1}^m \left\|x_k^i - x_*^i\right\|^2_{\X_i} \right] \\
&= \frac{1}{2\gamma_x} \E_k\left[\left\|x_{k+1}^{i_{k+1}} - x_*^{i_{k+1}}\right\|^2_{\X_{i_{k+1}}} + \sum_{i \neq i_{k+1}}^m \left\|x_{k+1}^i - x_*^i\right\|^2_{\X_i} - \sum_{i = 1}^m \left\|x_k^i - x_*^i\right\|^2_{\X_i} \right] \\
&= \frac{1}{2\gamma_x} \E_k\left[\left\|(1 - \alpha)x_k^{i_{k+1}} + \alpha \bar{x}_{k+1}^{i_{k+1}} - x_*^{i_{k+1}}\right\|^2_{\X_{i_{k+1}}}  - \left\|x_k^{i_{k+1}} - x_*^{i_{k+1}}\right\|^2_{\X_{i_{k+1}}} \right] \\
&= \frac{1}{2m \gamma_x} \sum_{\ell = 1}^m \left\|(1-\alpha)x_k^\ell + \alpha \bar x_{k+1}^\ell - x_*^\ell \right\|^2_{\X_\ell} - \left\|x_k^\ell - x_*^\ell\right\|^2_{\X_\ell} \\
&=\frac{1}{2 m \gamma_x} \|\alpha(\bar x_{k+1}-x_k) + x_k- x_*\|^2_{\X} - \frac{1}{2 m \gamma_x} \|x_k - x_*\|^2_{\X} \\
& = \frac{1}{2 m \gamma_x} \left(\|x_k - x_*\|^2_{\X} + 2\alpha \langle \bar x_{k+1} - x_k, x_k - x_*\rangle_{\X} + \alpha^2 \|\bar x_{k+1} - x_k\|^2_{\X} - \|x_k - x_*\|^2_{\X} \right) \\
& = \frac{\alpha }{m\gamma_x}\langle \bar x_{k+1} - x_k, x_k - \bar{x}_{k+1} + \bar{x}_{k+1} - x_*\rangle_{\X} +\frac{\alpha^2}{2m\gamma_x}\|\bar x_{k+1} - x_k\|^2_{\X} \\
& = -\frac{\alpha }{m} \left\langle \frac{1}{\gamma_x}(x_k - \bar x_{k+1}), \bar x_{k+1} - x_* \right\rangle_{\X} + \frac{\alpha^2 - 2\alpha}{2m\gamma_x}\|\bar x_{k+1} - x_k\|^2_{\X} \\
& = -\frac{\alpha }{m} \left\langle \bar F_{k+1}, \bar x_{k+1} - x_* \right\rangle_{\X} + \frac{\gamma_x(\alpha^2 - 2\alpha)}{2m}\|\bar F_{k+1}\|^2_{\X}
\end{align*}
\item We study the dual quantity, $\E_k\left[\|y_{k+1} - y_*\|^2_{\Y}\right]$, where the random extrapolation is applied. 
\begin{align*}
& \frac{1}{2\gamma_y} \E_k\left[\|y_{k+1} - y_*\|^2_{\Y} - \|y_k - y_*\|^2_{\Y}\right]\\
& = \frac{1}{2\gamma_y} \E_k\Big[\left\|(1 - \alpha) y_k + \alpha \bar y_{k+1} + \gamma_y \theta A(x_{k+1} - x_k) + \alpha \gamma_y \left(\nabla g_2(y_k) - \nabla g_2(\bar y_{k+1})\right) - y_* \right\|^2_{\Y} \\ 
& \qquad \hspace{9cm} - \|y_k - y_*\|^2_{\Y} \Big] \\
& = \frac{1}{2\gamma_y} \E_k\Bigg[\left\|(y_k - y_*) - \alpha \gamma_y \left(\frac{1}{\gamma_y} (y_k - \bar y_{k+1}) - \frac{\theta}{\alpha} A(x_{k+1} - x_k) - \left(\nabla g_2(y_k) - \nabla g_2(\bar y_{k+1}) \right) \right)\right\|^2_{\Y} \\ 
& \qquad \hspace{9cm} - \|y_k - y_*\|^2_{\Y}\Bigg] \\
& = \frac{1}{2\gamma_y} \E_k \left[\left\|y_k - y_* - \alpha\gamma_y \left(\bar G_{k+1} + A(\bar x_{k+1} - x_k) - \frac{\theta}{\alpha} A(x_{k+1} - x_k)\right)\right\|^2_{\Y} - \|y_k - y_*\|^2_{\Y}\right] \\
& = - \alpha \E_k \left[\left\langle \bar G_{k+1} + A(\bar x_{k+1} - x_k) - \frac{\theta}{\alpha} A( x_{k+1} - x_k), y_k - y_*\right\rangle_{\Y}\right] \\ 
& \qquad \hspace{3cm} + \frac{\alpha^2 \gamma_y}{2}\E_k\left[\left\|\bar G_{k+1}  + A(\bar x_{k+1} - x_k) - \frac{\theta}{\alpha} A(x_{k+1} - x_k)\right\|^2_{\Y}\right] \\
& \stackrel{(\ref{eqn:H_k+1_inner}, \ref{eqn:H_k+1_norm})}{=} -\alpha \left\langle \bar G_{k+1}, y_k - y_* \right \rangle_{\Y} + \frac{\alpha^2 \gamma_y}{2} \left( \left\|\bar G_{k+1}\right\|^2_{\Y} +  \frac{\theta^2}{\alpha^2} \E_k\left[\left\|A(x_{k+1} - x_k)\right\|^2_{\Y}\right] - \left\|A(\bar x_{k+1} - x_k) \right\|^2_{\Y}\right) \\
& \stackrel{ \eqref{eqn:E_x_tilde_norm}}{\leq} -\alpha \left\langle \bar G_{k+1}, y_k - y_* \right \rangle_{\Y} + \frac{\alpha^2 \gamma_y}{2} \left( \left\|\bar G_{k+1}\right\|^2_{\Y} +  m \sup_\ell \|A_\ell\|^2 \|\bar x_{k+1} - x_k\|^2_{\X} - \left\|A(\bar x_{k+1} - x_k) \right\|^2_{\Y}\right) \\
& =  -\alpha \left\langle \bar G_{k+1}, y_k - y_* \right \rangle_{\Y} + \frac{\alpha^2 \gamma_y}{2} \left( \left\|\bar G_{k+1}\right\|^2_{\Y} +  m \gamma_x^2 \sup_\ell \|A_\ell\|^2 \|\bar F_{k+1}\|^2_{\X} - \gamma_x^2 \left\|A \bar F_{k+1} \right\|^2_{\Y}\right) \\
& = -\alpha \left\langle \bar G_{k+1}, y_k - \bar y_{k+1} \right \rangle_{\Y} - \alpha \left\langle \bar G_{k+1}, \bar y_{k+1} - y_* \right \rangle_{\Y}  \\
&\qquad \qquad + \frac{\alpha^2 \gamma_y}{2} \left( \left\|\bar G_{k+1}\right\|^2_{\Y} +  m \gamma_x^2 \sup_\ell \|A_\ell\|^2 \|\bar F_{k+1}\|^2_{\X} - \gamma_x^2 \left\|A \bar F_{k+1} \right\|^2_{\Y}\right) 
\end{align*}
Now,
\begin{align*}
- \alpha & \left\langle  \bar G_{k+1}, y_k - \bar y_{k+1} \right\rangle_{\Y} \\
&= - \frac{\alpha \gamma_y}{2} \|\bar G_{k+1}\|^2_{\Y} - \frac{\alpha}{2\gamma_y} \|\bar y_{k+1} - y_k\|^2_{\Y} + \frac{\alpha \gamma_y}{2}\left\|\bar G_{k+1} - \frac{1}{\gamma_y} (y_k - \bar y_{k+1})\right\|^2_{\Y}  \\
& = - \frac{\alpha \gamma_y}{2} \|\bar G_{k+1}\|^2_{\Y} - \frac{\alpha}{2\gamma_y} \|\bar y_{k+1} - y_k\|^2_{\Y} + \frac{\alpha \gamma_y}{2} \left\|A(\bar x_{k+1} - x_k) + \nabla g_2(y_k) - \nabla g_2(\bar y_{k+1})\right\|^2_{\Y} \\
& \leq  - \frac{\alpha \gamma_y}{2} \|\bar G_{k+1}\|^2_{\Y} - \frac{\alpha}{2\gamma_y} \|\bar y_{k+1} - y_k\|^2_{\Y} + \alpha \gamma_y \left\|A (\bar x_{k+1} - x_k)\right\|^2_{\Y} \\ 
& \qquad \qquad + \alpha\gamma_y \|\nabla g_2(y_k) - \nabla g_2(\bar y_{k+1})\|^2_{\Y} \\
& \leq - \frac{\alpha \gamma_y}{2} \|\bar G_{k+1}\|^2_{\Y} + \left(\alpha\gamma_y L{_\nabla g_2}^2 - \frac{\alpha}{2 \gamma_y}\right) \|\bar y_{k+1} - y_k\|^2_{\Y} + \alpha\gamma_y \|A(\bar x_{k+1} - x_k)\|^2_{\Y} \\
& = - \frac{\alpha \gamma_y}{2} \|\bar G_{k+1}\|^2_{\Y} + \left(\alpha\gamma_y L{_\nabla g_2}^2 - \frac{\alpha}{2 \gamma_y}\right) \|\bar y_{k+1} - y_k\|^2_{\Y} + \alpha\gamma_y\gamma_x^2 \|A \bar F_{k+1}\|^2_{\Y} 
\end{align*}
Hence, 
\begin{align*}
& \frac{1}{2\gamma_y} \E_k\left[\|y_{k+1} - y_*\|^2_{\Y} - \|y_k - y_*\|^2_{\Y}\right]\\
&= - \alpha \left\langle \bar G_{k+1}, \bar y_{k+1} - y_* \right \rangle_{\Y}  - \frac{\alpha \gamma_y}{2} \|\bar G_{k+1}\|^2_{\Y} + \left(\alpha\gamma_y L{_\nabla g_2}^2 - \frac{\alpha}{2 \gamma_y}\right) \|\bar y_{k+1} - y_k\|^2_{\Y} \\ 
&\qquad \qquad + \alpha\gamma_y\gamma_x^2 \|A \bar F_{k+1}\|^2_{\Y} + \frac{\alpha^2 \gamma_y}{2} \left( \left\|\bar G_{k+1}\right\|^2_{\Y} +  m\gamma_x^2 \sup_\ell \|A_\ell\|^2 \|\bar F_{k+1} \|^2_{\X} - \gamma_x^2\left\|A \bar F_{k+1}\right\|^2_{\Y} \right) \\
&= - \alpha \left\langle \bar G_{k+1}, \bar y_{k+1} - y_* \right \rangle_{\Y}  + \frac{\alpha \gamma_y}{2} (\alpha -1 ) \|\bar G_{k+1}\|^2_{\Y} + \left(\alpha\gamma_y L{_\nabla g_2}^2 - \frac{\alpha}{2 \gamma_y}\right) \|\bar y_{k+1} - y_k\|^2_{\Y} \\ 
&\qquad \qquad + \alpha\gamma_y\gamma_x^2 \left(1 - \frac{\alpha}{2}\right)\|A \bar F_{k+1}\|^2_{\Y} + \frac{m \alpha^2 \gamma_x^2 \gamma_y}{2} \sup_\ell \|A_\ell\|^2 \|\bar F_{k+1}\|^2_{\X}
\end{align*}
\item We sum the inequality on primal and dual space and weigh them in order to apply the weak MVI inequality.
\begin{align*}
\frac{m}{2\gamma_x} & \E_k \left[\|x_{k+1} - x_*\|^2_{\X} - \|x_k - x_*\|^2_{\X} \right] + \frac{1}{2\gamma_y} \E_k\left[\|y_{k+1} - y_*\|^2_{\Y} - \|y_k - y_*\|^2_{\Y} \right] \\
&= -\alpha \left\langle \bar F_{k+1}, \bar x_{k+1} - x_* \right\rangle_{\X} 
+ \frac{\gamma_x(\alpha^2 - 2\alpha)}{2}\|\bar F_{k+1}\|^2_{\X}  - \alpha  \left\langle \bar G_{k+1}, \bar y_{k+1} - y_* \right \rangle_{\Y} \\
&\qquad + \frac{\alpha \gamma_y}{2} (\alpha - 1 ) \|\bar G_{k+1}\|^2_{\Y} + \left(\alpha\gamma_y L^2_{\nabla g_2} - \frac{\alpha}{2 \gamma_y}\right) \|\bar y_{k+1} - y_k\|^2_{\Y} \\
&\qquad + \alpha\gamma_y\gamma_x^2 \left(1 - \frac{\alpha}{2}\right)\|A \bar F_{k+1}\|^2_{\Y} + \frac{m \alpha^2 \gamma_x^2 \gamma_y}{2} \sup_\ell \|A_\ell\|^2 \|\bar F_{k+1}\|^2_{\X} \\
&= -\alpha \left(\left\langle \bar F_{k+1}, \bar x_{k+1} - x_* \right\rangle_{\X} + \left\langle \bar G_{k+1}, \bar y_{k+1} - y_* \right \rangle_{\Y}  \right) + \frac{\alpha \gamma_y}{2} (\alpha - 1 ) \|\bar G_{k+1}\|^2_{\Y} \\
&\qquad  + \alpha \gamma_x \left(\frac{\alpha}{2} - 1 + \frac{m \alpha \gamma_x \gamma_y}{2} \sup_\ell \|A_\ell\|^2 \right) \|\bar F_{k+1}\|^2_{\X}  \\
&\qquad + \left(\alpha\gamma_y L^2_{\nabla g_2} - \frac{\alpha}{2 \gamma_y}\right) \|\bar y_{k+1} - y_k\|^2_{\Y} + \alpha\gamma_y\gamma_x^2 \left(1 - \frac{\alpha}{2}\right)\|A \bar F_{k+1}\|^2_{\Y} \\
&\leq -\alpha\rho\left(\left\|\bar F_{k+1}\right\|^2_{\X} + \left\|\bar G_{k+1}\right\|^2_{\Y} \right) + \frac{\alpha \gamma_y}{2} (\alpha - 1 ) \|\bar G_{k+1}\|^2_{\Y} \\
&\qquad  + \alpha \gamma_x \left(\frac{\alpha}{2} - 1 + \frac{m \alpha \gamma_x \gamma_y}{2} \sup_\ell \|A_\ell\|^2 + \gamma_x \gamma_y \left(1 - \frac{\alpha}{2}\right)\|A\|^2\right) \|\bar F_{k+1}\|^2_{\X}  \\
&\qquad + \left(\alpha\gamma_y L^2_{\nabla g_2} - \frac{\alpha}{2 \gamma_y}\right) \|\bar y_{k+1} - y_k\|^2_{\Y}  \\
&= - \alpha \left(\rho + \gamma_x \left(1 - \frac{\alpha}{2} \right) - \gamma_x^2 \gamma_y \left(\frac{m \alpha}{2} \sup_\ell \|A_\ell\|^2 + \left(1 - \frac{\alpha}{2}\right)\|A\|^2\right)\right) \|\bar F_{k+1}\|^2_{\X}  \\
&\qquad - \alpha \left(\rho + \frac{\gamma_y}{2} (1 - \alpha)\right) \|\bar G_{k+1}\|^2_{\Y} + \left(\alpha\gamma_y L^2_{\nabla g_2} - \frac{\alpha}{2 \gamma_y}\right) \|\bar y_{k+1} - y_k\|^2_{\Y}  \\
&= - \alpha \left(C_x \|\bar F_{k+1}\|^2_{\X} + C_y \|\bar G_{k+1}\|^2_{\Y} \right) + \left(\alpha\gamma_y L^2_{\nabla g_2} - \frac{\alpha}{2 \gamma_y}\right) \|\bar y_{k+1} - y_k\|^2_{\Y}  \\
&\leq - \alpha \min(C_x, C_y) \|\bar D_{k+1}\|^2_{\Z} + \left(\alpha\gamma_y L^2_{\nabla g_2} - \frac{\alpha}{2 \gamma_y}\right) \|\bar y_{k+1} - y_k\|^2_{\Y}
\end{align*}
\end{enumerate}
Thus, taking $\alpha$ and the step sizes, $\gamma_x$ and $\gamma_y$, such that \eqref{eqn:SPDHG_param_cond} holds, we get:
\begin{align*}
    &\|\bar D_{k+1}\|_{\Z} \leq  \\
    &\frac{1}{2\alpha C} \left(\frac{m}{\gamma_x} \E_k\left[\|x_k - x_*\|^2_{\X} - \|x_{k+1} - x_*\|^2_{\X}\right] + \frac{1}{\gamma_y} \E_k\left[\|y_k - y_*\|^2_{\Y} - \|y_{k+1} - y_*\|^2_{\Y}\right]\right)
\end{align*}
Taking the full expectation, summing from $k = 0$ to $k = K-1$, and telescoping implies the result. \qed
\end{proof}
\section{Numerical Experiments} \label{sec:numerical_exp}
In this final section, we present our numerical experiments, which aim to illustrate and validate our theoretical findings by demonstrating the efficiency of our proposed algorithms. We consider the saddle point problem \eqref{eqn:SPP_SPDHG} with $g = 0$ which reduces to:
\begin{equation} \tag{N-SPP} \label{eqn:SPP_numerical}
\min_{x \in \X} \max_{y \in \Y} ~ \Lag(x, y) = \sum_{i = 1}^m f_i(x_i) + \underbrace{\langle A x, y \rangle - g_2(y)}_{\varphi(x, y)}
\end{equation}

\vspace{0.2cm}
\noindent
\textbf{Nonconvex-nonconcave experiments:}
We begin with two experiments involving nonconvex-nonconcave regression problems:
\begin{itemize}
\item Logistic regression with squared loss.
\item Perceptron regression with ReLU activation.
\end{itemize}
For both problems, we compare our proposed algorithms—Algorithm~\ref{alg:NC-PDHG} (NC-PDHG) and Algorithm~\ref{alg:NC-SPDHG} (NC-SPDHG)—with two state-of-the-art baselines:
\begin{enumerate}
\item Constrained Extra-Gradient+ (CEG+) \cite{EG+}, with constant step sizes.
\item Augmented Lagrangian Method (ALM) from \cite{ALM}.
\end{enumerate}
Performance is measured in terms of the total number of proximal evaluations required to reach a $\tau = 10^{-7}$-stationary point.

\vspace{0.2cm}
\noindent
\textbf{Convex-concave experiment.}
To demonstrate that our method also performs well in the convex-concave setting, we include a third experiment involving a convex-concave least-squares problem. We compare NC-SPDHG with SAGA \cite{saga}, a leading algorithm in smooth convex optimization.

For each experiment, we provide or calculate the necessary attributes for each algorithm, like proximals, gradients, smoothness constants, etc..

\vspace{0.2cm}
\noindent
\textbf{Parameters selection:} The parameters for the different algorithms are chosen as follows: 
\begin{enumerate}
    \item For NC-PDHG and NC-SPDHG, we select parameters according to Subsections~\ref{subsec:parameters_NC_PDHG} and~\ref{subsec:parameters_NC_SPDHG}, respectively.
    \item For CEG+ with constant step sizes, we follow the prescription in \cite[Corollary 3.2]{EG+}. Specifically:
    \begin{itemize}
        \item $\gamma_{\ceg} = \frac{1}{L_{\nabla \varphi}}$ with $L_{\nabla \varphi} = \sqrt{2} \left(L_{\nabla g_2} + \|A\|\right)$.
        \item $\delta = \rho$. 
        \item $\alpha_{\ceg} = 1 + \frac{2\delta}{\gamma_{\ceg}} - \varepsilon_{\ceg}$ for some $\varepsilon_{\ceg} > 0$.
    \end{itemize}
    \item For ALM, we employ a suitable subroutine to solve the inner optimization problem: 
    \begin{equation*}
        \min_{x \in \X} \Lambda_{\mu}(x, y_k) := \Lag(x, y_k) + \frac{\mu}{2} \|Ax\|^2 
    \end{equation*}
    The parameters for the subroutine are chosen according to its theoretical guidelines. We arbitrarily set $\mu = 0.5$.
\end{enumerate}
\textbf{Weak MVI parameter selection:} The weak MVI parameter, $\rho$, is generally hard to find or even to estimate. To address this, we test several feasible estimates of $\rho$ and observe the convergence behavior of the considered algorithms. This analysis is conducted for the \textit{logistic regression with squared loss} experiment described in Subsection \ref{subsec:logistic_regression}. Table~\ref{tab:rho_values} summarizes the tested values of $\rho$ and the corresponding convergence status of each algorithm. Based on these results, we conclude that $\rho = -2 \times 10^{-3}$ is the largest value for which convergence is observed across all algorithms while being feasible, i.e. $\rho > \frac{1}{2\sqrt{2} }\max \left\{\frac{-1}{\|A\|}, \frac{-1}{L_{\nabla g_2}} \right\}$ .

\begin{table}[H]
    \centering
    \begin{tabular}{||c|c|c|c|c|c|c|c||}
    \hline
        \textbf{Estimate of $\rho$} & $0$ & $-10^{-5}$ & $-10^{-4}$ & $-10^{-3}$ & $-2 \times 10^{-3}$ &  $-5 \times 10^{-3}$ & $-9 \times 10^{-3}$ \\
        \hline
        \hline
         \textbf{NC-PDHG} & \cmark & \cmark & \cmark & \cmark &\cmark & \cmark & \cmark (s) \\
         \hline
         \textbf{NC-SPDHG} & \cmark & \cmark & \cmark & \cmark &\cmark & \cmark & \cmark (s) \\
         \hline
         \textbf{CEG+}  & \xmark & \xmark & \xmark & \xmark & \cmark & \cmark & \cmark (s) \\
         \hline
         \textbf{ALM}  & \cmark & \cmark & \cmark & \cmark &\cmark & \cmark & \cmark \\
         \hline
    \end{tabular}
    \caption{Testing the convergence of the considered algorithms for several estimates of $\rho$. \cmark marks convergence, \xmark ~ marks divergence, and (s) marks slow convergence}
    \label{tab:rho_values}
\end{table}

\vspace{0.2cm}
\noindent
\textbf{Stopping criterion:} Given a tolerance $\tau > 0$, we stop the algorithm when identifying a $\tau$-stationary point, as presented in Section \ref{sec:assumptions}.

\vspace{0.2cm}
\noindent
\textbf{Dataset:} All experiments are conducted on the \href{https://www.csie.ntu.edu.tw/~cjlin/libsvmtools/datasets/regression.html}{\texttt{pyrim\_scale}} dataset \cite{LIBSVM}, where each data sample is represented as $(B_i, b_i)_{1\leq i \leq m} \in \R^n \times \R$. The dataset has $m = 74$ samples and $n = 27$ features.

\vspace{0.2cm}
\noindent
\textbf{Computer configurations:} 
All numerical results are obtained on a macOS machine equipped with an Apple M1 Pro processor and 16 GB RAM. The experiments are implemented in Python 3.10.9.
 
\subsection{Logistic regression with squared loss} \label{subsec:logistic_regression}

We consider the following logistic regression problem with a squared loss:
\begin{equation*}
    \min_{ \\\mu \in \R^n} \sum_{i = 1}^{m} \left(\sigma\left(B_i^\top \mu - b_i\right) - 0.5\right)^2
\end{equation*}
where $B_i$ is the $i^{\text{th}}$ row of $B \in \R^{m \times n}, b \in \R^m$, and $$\sigma(u) = \frac{1}{1 + \exp(-u)}$$
To  solve this problem using our algorithms, we first reformulate it as a saddle point problem. We introduce a new variable $\nu \in \R^m$ defined by $\nu_i = B_i^\top \mu - b_i, ~\forall i \in \{1, \dots, m\}$. That is: 
\begin{equation} \label{eqn:C-sigmoid}
    \begin{split}
        \min_{y \in \R^{n + m}} ~ &\sum_{i = 1}^{m} \left(\sigma (\nu_i) - 0.5\right)^2 \\ 
        \text{s.t.} ~&~ Ay = b
    \end{split}
\end{equation}
where 
\begin{align*}
    y = \begin{pmatrix} \mu \\ \nu \end{pmatrix}_{(n+m) \times 1} && \text{and} && A = \begin{bmatrix} B && -\Id_m \end{bmatrix}_{m \times (n + m)}
\end{align*}
Now, we can write the saddle point problem of (\ref{eqn:C-sigmoid}) as follows: 
\begin{align*}
    &\min_{y \in \R^{n + m}} \max_{x \in \R^m}  \sum_{i = 1}^{m} \left(\sigma(\nu_i) - 0.5\right)^2 + \left\langle A y - b, x \right\rangle  \\ 
    \equiv & \min_{y \in \R^{n + m}} \max_{x \in \R^m} - \left(\sum_{i = 1}^{m} b_i x_i + \left\langle - A y, x \right\rangle - \sum_{i = 1}^{m} \left(\sigma(\nu_i) - 0.5\right)^2 \right)\\ 
    \equiv & - \max_{y \in \R^{n + m}} \min_{x \in \R^m} \left(\sum_{i = 1}^{m} b_i x_i + \left\langle -A^\top x, y \right\rangle - \sum_{i = 1}^{m} \left(\sigma(\nu_i) - 0.5\right)^2 \right) \stepcounter{equation}\tag{SPP-sigmoid}\label{eqn:SPP-sigmoid}
\end{align*}
Hence, by comparing \eqref{eqn:SPP-sigmoid} with the saddle point problem \eqref{eqn:SPP_numerical}, we identify, for all $i \in \{1, \dots, m\}$
\begin{align*}
    f_i(x_i) = b_ix_i && g_2(y) = \sum_{i = 1}^m \left(\sigma(\nu_i) - 0.5\right)^2 
\end{align*}

\subsubsection{Problem analysis} \label{subsubsec:analysis_sigmoid}
\begin{enumerate}
    \item The proximal of $f_i, ~\forall i \in \{1, \dots, m\}$, is given by: $\prox_{\gamma_x f_i} (x_i) = x_i - \gamma_x b_i$
\item The gradient of $g_2$ is:
\begin{align*}
    \nabla g_2(y) &= \begin{bmatrix} \nabla_\mu g_2(y) \\ \nabla_\nu g_2(y)\end{bmatrix} = \begin{bmatrix}
        0 \\ \left(\frac{\partial g_2}{\partial \nu_i}(y)\right)_{1 \leq i \leq m} 
    \end{bmatrix} \\ &= \begin{bmatrix}
            0 \\ \left(2\sigma(\nu_i)(\sigma(\nu_i) - 0.5)(1 - \sigma(\nu_i))\right)_{1 \leq i \leq m}
        \end{bmatrix}
\end{align*}
\item The Lipschitz constant of the gradient of $g_2$ is: $$L_{\nabla g_2} = \left\|\nabla^2 g_2(y)\right\|_{op} = \mathrm{spr}\left(\nabla^2 g_2 \right) = \max\{|\lambda_1|, \dots, |\lambda_{m}|\}$$
where $\mathrm{spr(\nabla^2 g_2)}$ denotes the spectral radius of $\nabla^2 g_2$, and for all $i, j \in \{1, \dots, m\}$, we have: 
\begin{align*}
    \nabla^2_\mu g_2(y) = \frac{\partial^2 g_2}{\partial \mu \partial \nu_i}(y) = \frac{\partial^2 g_2}{\partial \nu_i \partial \mu_j} = 0 && \frac{\partial^2 g_2}{\partial \nu_i \partial \nu_j} = 0~~ \forall i \neq j
\end{align*}
and
\begin{align*}
    \frac{\partial^2 g_2}{\partial \nu^2_i} (y) &= \frac{\partial}{\partial \nu_i} \left(2\sigma(\nu_i)(\sigma(\nu_i) - 0.5)(1 - \sigma(\nu_i))\right) \\
    &= \sigma(\nu_i) (1 - \sigma(\nu_i)) \left[-6 \sigma^2(\nu_i) + 6 \sigma(\nu_i) -1\right]
\end{align*}
Hence, $\nabla^2 g_2$ has at most $m$ nonzero entries on its main diagonal, given by $\frac{\partial^2 g_2}{\partial \nu^2_i}, ~\forall i \in \{1, \dots, m\}$. Therefore, the spectral radius $\mathrm{spr}\left(\nabla^2 g_2\right)$ corresponds to the maximum absolute among these diagonal entries. Note that, for each $\nu_i \in \R$ we have: $$\frac{\partial^2 g_2}{\partial \nu_i^2} =  \sigma(\nu_i) (1 - \sigma(\nu_i)) \left[-6 \sigma^2(\nu_i) + 6 \sigma(\nu_i) -1\right]$$ 
which is a one-dimensional expression that can be interpreted as a degree-4 polynomial in $\sigma(\nu_i)$. This function has a local maximum at $\sigma(\nu_i) = 0.5$ and local minima at $\sigma(\nu_i) = \frac{1}{2} \pm \frac{1}{\sqrt{6}}$. Substituting and taking the absolute maximum yields $\mathrm{spr}\left(\nabla^2 g_2\right) = 0.125$.

\item ALM analysis: From \eqref{eqn:C-sigmoid}, we define the following augmented Lagrangian 
\begin{equation*}
    \Lambda_\mu (y, x) = \sum_{i = 1}^m (\sigma(\nu_i) - 0.5)^2 + \langle Ay - b, x \rangle + \frac{\mu}{2} \|Ay - b\|^2
\end{equation*}
Then, we employ the Gradient Descent (GD) algorithm to solve the inner optimization problem, $\min_{y} \Lambda_\mu (y, x_k)$. The step size is chosen as $\gamma_{\alm} = \frac{1}{L_{\alm}}$ where $L_{\alm}$ denotes the Lipschitz constant of the gradient of $\Lambda_\mu (y, x_k)$. That is, for a given $x_k$ and for any $y, y' \in \R^{m+n}$, we have: 
\begin{align*}
    \left\|\nabla \Lambda_\mu (y, x_k) - \nabla \Lambda_{\mu}(y', x_k) \right\|
    &= \left\| \nabla g_2(y) - \nabla g_2(y') + \mu A^\top A (y - y') \right\| \\
    &\leq \left\|\nabla g_2(y) - \nabla g_2(y') \right\| + \mu \left\| A^\top A (y - y')\right\| \\ 
    &\leq L_{\nabla g_2} \|y - y'\| + \mu  \left\|A^\top A \right\|_{op} \|y - y'\|
\end{align*}
Thus, $L_{\alm} = L_{\nabla g_2} + \mu \left\|A^\top A \right\|_{op}$.


\item KKT error: considering \eqref{eqn:SPP-sigmoid}, the KKT error is defined as: 
\begin{equation*}
    \K(x, y) = \left\| -B^\top x - \nabla g(y)\right\|^2 + \left\|By - b\right\|^2
\end{equation*}
\end{enumerate}

\subsection{Perceptron regression with ReLU} \label{subsec:perceptron}
 
We are interested in the following perceptron regression problem with a ReLU activation function: 
\begin{equation*}
    \min_{w \in \R^n} \left\|b - \hat b(w)\right\|^2
\end{equation*}
where for each $i \in \{1, \dots, m\}$,  
$$\hat b_i (w) = r \left(\langle w, B_i \rangle \right) := \max(0, \langle w, B_i, \rangle )$$ Equivalently, we would like to solve the following saddle point problem:
\begin{align*}
&\min_{l, w} \|l - b\|^2 \hspace{0.35cm} \text{s.t.} \hspace{0.35cm} l_i = r \left(\langle w, B_i \rangle\right), ~ \forall i \in \{1, \dots, m\} \\
\equiv &\min_{w, l, u} \|l - b \|^2 \hspace{0.25cm} \text{s.t.} \hspace{0.25cm} l_i = r(u_i)~~ \text{and}~~ u_i = w^\top B_i, ~ \forall i  \\
\equiv &\min_{w, l, u } \|l -  b \|^2 + \sum_{i=1}^m \imath_{\gp(r)}(u_i, l_i) \hspace{0.25cm} \text{s.t.} \hspace{0.25cm}  u  = B w
\end{align*}
where $\gp{r} = \{(u,l) \;:\; l = r(u)\}$ is the graph of $r$ and $\imath_{\gp(r)}$ is the indicator of this set.
We are kind of done up to here, however, to obtain a more tractable proximal of the objective function, we further introduce an auxiliary variable $\lambda$ such that $\lambda = l$. That is:
\begin{align*}
&\min_{w, l, u, \lambda } \|\lambda -  b \|^2 + \sum_{i=1}^m \imath_{\gp(r)}(u_i, l_i) \hspace{0.25cm} \text{s.t.} \hspace{0.25cm}  u  = B w ~~ \text{and} ~~ \lambda = l \\
\equiv &\min_{w, l, u, \lambda} \max_{\mu, \nu}\|\lambda - b\|^2 + \sum_{i=1}^m \imath_{\gp(r)}(u_i, l_i) + \langle Bw - u, \mu\rangle + \langle l - \lambda, \nu \rangle \\
\equiv &\min_{w, l, u, \lambda} \max_{\mu, \nu}\|\lambda - b\|^2 + \sum_{i=1}^m \imath_{\gp(r)}(u_i, l_i) + \left\langle \underbrace{\begin{bmatrix} B\;, & -\mathbf{I}_{m \times m} \end{bmatrix}}_{A_1} \begin{pmatrix} w \\ u \end{pmatrix}, \mu \right\rangle \\
& \qquad \hspace{6cm} + \left\langle \underbrace{\begin{bmatrix}\mathbf{I}_{m\times m}\;, & -\mathbf{I}_{m \times m}\end{bmatrix}}_{A_2} \begin{pmatrix} l \\ \lambda \end{pmatrix}, \nu \right\rangle \\
\equiv &\min_{x} \max_{y} \|\lambda - b\|^2 + \sum_{i=1}^m \imath_{\gp(r)}(u_i, l_i) + \langle A x, y \rangle \stepcounter{equation}\tag{SPP-perceptron}\label{eqn:SPP-perceptron}
\end{align*}
where $w \in \R^n$ and $u, \lambda, l, \mu, \nu \in \R^m$
\begin{align*}
    x = \begin{pmatrix} w \\ u \\ l \\ \lambda \end{pmatrix}_{(n + 3m) \times 1} && A = \begin{bmatrix} A_1 & \mathbf{0}_{m \times 2m} \\ \mathbf{0}_{m\times (m+n)} & A_2\end{bmatrix}_{2m \times (n + 3m)} && v = \begin{pmatrix} \mu \\ \nu \end{pmatrix}_{2m \times 1}
\end{align*}
Hence, by comparing \eqref{eqn:SPP-perceptron} with the saddle point problem \eqref{eqn:SPP_numerical}, we identify $g_2 = 0$, and $f(x) = \sum_{i = 1}^m h(\lambda_i) + k(u_i, l_i)$ where for any $i \in \{1, \dots, m\}$:
\begin{align*}
h(\lambda_i) := (\lambda_i - y_i)^2 && \text{and} &&
    k(u_i, l_i) := \imath_{\gp(r)}(u_i, l_i)
\end{align*}

\subsubsection{Problem analysis} \label{subsubsec:analysis_perceptron}
\begin{enumerate}
    \item Proximal of $f$: Thanks to the separability of $f$,  its proximal operator can be computed component-wise as follows: 
\begin{equation*}
\prox_{\gamma_xf}(x) = \left(\left(\prox_{0}(w_i)\right)_{1 \leq i \leq n}, \left(\prox_{\gamma_x k}(u_j, l_j)\right)_{1 \leq j \leq m}, \left(\prox_{\gamma_x h}(\lambda_\ell)\right)_{1 \leq \ell \leq m}\right)
\end{equation*}
where for all $i \in \{1, \dots, n\}$ and for all $j, \ell \in \{1, \dots, m\}$
\begin{align*}
    \prox_0(w_i) &= w_i \\ 
    \prox_{\gamma_x k}(u_j, l_j) &= \prox_{\gamma_x \imath_{\gp(r)}}(u_j, l_j) = \proj_{\gp(r)}(u_j, l_j) \\
    \prox_{\gamma_x h}(\lambda_\ell) &= \prox_{\gamma_x (. - y_\ell)^2}(\lambda_\ell) = \frac{\gamma_x}{2\gamma_x + 1} \left(2y_\ell + \frac{1}{\gamma_x}\lambda_\ell \right) 
\end{align*}
The projection onto the graph of the ReLU function is a non-trivial exercise that needs a fair amount of calculations. Therefore, we present its formula here  along with a geometric illustration in Fig. \ref{fig:relu_proj}, and defer the detailed derivation to Appendix \ref{apdx:proj}. For any $j \in \{1, \dots, m\}$, we denote $(\tilde u_j, \tilde l_j) = \proj_{\gp(r)}(u_j, l_j)$ such that $\tilde l_j = r(\tilde u_j)$ and 
\begin{equation} \label{eqn:relu_proj}
\tilde u_j = \begin{cases}
        0, & u_j \geq 0, u_j + l_j \leq 0 \\
        \frac{u_j + l_j}{2}, & u_j \geq 0, u_j + l_j > 0 \\
        \frac{u_j + l_j}{2}, & u_j < 0, \left(1 + \sqrt{2}\right)u_j + l_j  > 0 \\
        u_j, & u_j < 0, \left(1 + \sqrt{2}\right)u_j + l_j \leq 0 
    \end{cases}
\end{equation}
\begin{figure}[H]
    \centering
\begin{tikzpicture}
    \begin{axis}[
        axis lines=middle,
        xlabel={$u$},
        ylabel={$l$},
        samples=100,
        domain=-2.2:2.2,
        ymin=-2.2, ymax=2.2,
        xmin=-2.2, xmax=2.2,
        clip=false,
        enlargelimits=true,
        legend pos=south east
    ]
    \addplot[name path=A1, domain=-2:0] {-x};
    \addplot[name path=B1, domain=-2:0] {-2};  
    \addplot[red!30, opacity=0.5] fill between[of=A1 and B1];

    \addplot[name path=A11, domain=-2:(-2/(1+sqrt(2)))] {2};
    \addplot[name path=B11, domain=-2:0] {-x};  
    \addplot[red!30, opacity=0.5] fill between[of=A11 and B11];
    
    \addplot[name path=A2, domain=0:2] {-x};
    \addplot[name path=B2, domain=0:2] {-2};  
    \addplot[green!30, opacity=0.5] fill between[of=A2 and B2];
    \addplot[name path=A3, domain=0:2] {max(0, x)};
    \addplot[name path=B3, domain=0:2] {-x};  
    \addplot[blue!30, opacity=0.5] fill between[of=A3 and B3];
    
    \addplot[name path=A4, domain=0:2] {max(0, x)};
    \addplot[name path=B4, domain=0:2] {2};
    \addplot[blue!30, opacity=0.5] fill between[of=A4 and B4];

    \addplot[name path=A5, domain=(-2/(1+sqrt(2)):0]{-(1+sqrt(2))*x};
    \addplot[name path=B5, domain=(-2/(1+sqrt(2)):0] {2};
    \addplot[blue!30, opacity=0.5] fill between[of=A5 and B5];

    \addplot[thick,blue] {max(0,x)} node[right] {ReLU$(u)$};
    \addplot[thick,red] {-x} node[right] {$u + l = 0$};
    \addplot[thick, brown, domain=-1:0] {-(1 + sqrt(2)) * x} node[above] at (-1, 2.3) {\scriptsize $\left(1+\sqrt{2}\right)u + l = 0$};
    
    \draw[red, thick, ->] (-1.5,-1.5) -- (-1.5,0);
    \draw[green, thick, ->] (0.5,-1.5) -- (0,0);
    \draw[blue, thick, ->] (1,1.5) -- (1.25,1.25);
    \draw[blue, thick, ->] (-0.1,1.5) -- (0.7,0.7);
    \draw[red, thick, ->] (-0.5,1) -- (-0.5,0);
    \draw[blue, thick, ->] (1.5, -0.5) -- (0.5,0.5);

    \fill[black] (-1.5,-1.5) circle (2pt);
    \fill[black] (0.5,-1.5) circle (2pt);
    \fill[black] (1,1.5) circle (2pt);
    \fill[black] (-0.1,1.5) circle (2pt);
    \fill[black] (-0.5,1) circle (2pt);
    \fill[black] (1.5, -0.5) circle (2pt);
    \end{axis}
\end{tikzpicture}
    \caption{Projection onto the graph of ReLU: illustrating the different cases summarized in formula \eqref{eqn:relu_proj}}
    \label{fig:relu_proj}
\end{figure}
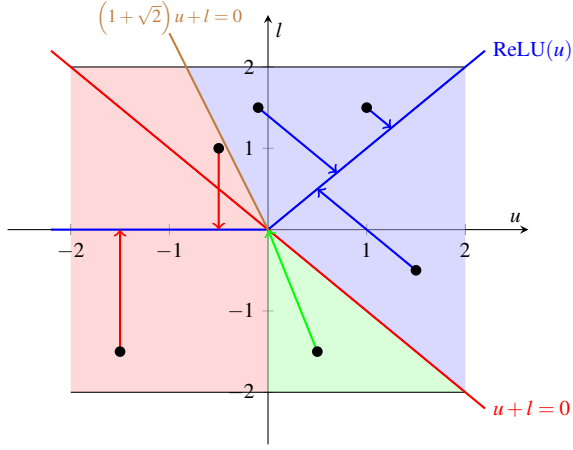

\item ALM analysis: from \eqref{eqn:SPP-perceptron}, we define the following augmented Lagrangian 
\begin{equation*}
    \Lambda_\mu (x, y) = \|\lambda - b\|^2 + \imath_{\gp(r)}(u, l) + \langle Ax, y \rangle + \frac{\mu}{2} \|Ax\|^2
\end{equation*}
Then, we employ the Proximal Gradient Descent (PGD) algorithm to solve the inner optimization problem: $\min_{x} \Lambda_\mu (x, y_k)$. We choose $\gamma_{\alm} = \frac{1}{L_{\alm}}$ where $L_{\alm}$ denotes the Lipschitz constant of the gradient of $\psi(x, y_k) := \langle Ax, y_k \rangle + \frac{\mu}{2} \|Ax\|^2$. That is, for a given $y_k$ and for any $x \in \R^{n + 3m}$, we have: 
\begin{align*}
    L_{\alm} = \left\|\nabla^2 \psi(x, y_k)\right\| = \mu \left\| A^\top A\right\|
\end{align*}
\item KKT error: considering \eqref{eqn:SPP-perceptron}, the Lagrangian is defined as: 
\begin{equation*}
    \Lag (x, y) = \|\lambda - b \|^2 + \imath_{\mathrm{gp}(r)} (u, l) + \langle Ax, y \rangle 
\end{equation*}
Hence, the KKT error is defined as:
\begin{align*}
    \K(x, y) &= \left\| \frac{\partial L}{\partial w} (x, y) \right\|^2 +  \left\| \frac{\partial L}{\partial (u, l)} (x, y) \right\|^2_0 +   \left\| \frac{\partial L}{\partial \lambda} (x, y) \right\|^2  + \left\| \frac{\partial L}{\partial y} (x, y) \right\|^2 
    \\
    &= \left\| A^\top_w y\right\|^2 + \left\| \partial_C \imath_{\mathrm{gp}(r)} (u, l) + \begin{bmatrix} A^\top_u y & A^\top_l y\end{bmatrix}_{n \times 2} \right\|^2_0 + \left\| 2(\lambda - b) + A^\top_\lambda y\right\|^2 + \left\| Ax\right\|^2 \\ 
    &= \left\| A^\top_w y\right\|^2 + \left\|\nor_{\mathrm{gp}(r)} (u, l) + \begin{bmatrix} A^\top_u y & A^\top_l y\end{bmatrix}_{n \times 2} \right\|^2_0 + \left\| 2(\lambda - b) + A^\top_\lambda y\right\|^2 + \left\| Ax\right\|^2
\end{align*}
where $A^\top_w, A^\top_u, A^\top_l$ and $A^\top_\lambda$ represent the sub-matrix formed from the corresponding rows in $A^\top$ to the vectors $w, u, l$ and $\lambda$, respectively. 

Using the definition of the Clarke normal cone \eqref{eqn:normal_cone}, we obtain, for each $j \in \{1, \dots, n\}$: 
\begin{equation*}
    \nor_{\mathrm{gp}(r)} (u_j, l_j) = \begin{cases}
        c_1 (0, u_j), &  u_j < 0 \\ 
        \R^2, & u_j = 0 \\ 
        c_2 (u_j, -u_j), & u_j > 0
    \end{cases}
\end{equation*}
Then, we substitute this formula in the infimal norm, $\left\| \nor_{\mathrm{gp}(r)} (u, l) + \begin{bmatrix}A^\top_u y & A^\top_l y\end{bmatrix}\right\|_0^2$, and find the element with the minimal norm. For any $j \in \{1, \dots, n\}$, we have three simple cases: When $u_j = 0$, it's a trivial projection on $\R^2$. When $u_j \neq 0$ we project on two different lines, which implies 
\begin{align*} c_1 = \frac{-\left\langle A^\top_{l_j}, y \right\rangle}{u_j} && \text{and} && c_2 = \frac{\left\langle A^\top_{l_j} - A^\top_{u_j}, y \right\rangle}{2u_j} \end{align*}
Hence, 
\begin{equation*}
    \left\| \nor_{\mathrm{gp}(r)} (u, l) + \begin{bmatrix}A^\top_u y & A^\top_l y\end{bmatrix}\right\|_0^2 = \sum_{j = 1}^n \begin{cases}
        \langle A^\top_{u_j} , y \rangle^2, & u_j < 0 \\
        0, & u_j = 0 \\ 
        \frac{1}{2} \left\langle A^\top_{u_j} + A^\top_{l_j}, y \right\rangle^2, & u_j > 0
    \end{cases}
\end{equation*}
\end{enumerate}

\subsection{Figures} \label{subsec:figs}
 
In this subsection, we present two main figures about our both experiments: logistic regression with squared loss, and perceptron regression with ReLU activation. Estimating $\rho = -2 \times 10^{-3}$, we outline the parameters choice in Table \ref{tab:params} for both experiments, based on the analysis done in sub-sub-sections \ref{subsubsec:analysis_sigmoid} and \ref{subsubsec:analysis_perceptron}.
\begin{table}[H]
    \centering
    \begin{tabular}{||c|c|c|c||}
    \hline
    \textbf{Experiment} & \textbf{Algorithm} & \textbf{Step sizes} & \textbf{Extrapolation parameters} \\
    \hline
    \hline
    \multirow{4}{*}{\shortstack{\textbf{Logistic} \\ \textbf{regression}}}
        & NC-PDHG & $\begin{aligned}c = 0.4 && \gamma_x = 0.024 && \gamma_y = 0.015\end{aligned}$ & $\alpha = 0.73$\\
        & NC-SPDHG & $\begin{aligned}c = 0.1 && \gamma_x = 0.054 && \gamma_y = 0.007\end{aligned}$ & $\begin{aligned} \theta = m = 74 && \alpha = 0.38 \end{aligned}$\\
        & CEG+ & $\gamma_{\ceg} = 0.02$ & $\begin{aligned} \delta_{\ceg} = -0.004 && \alpha_{\ceg} = 0.79 \end{aligned}$\\
        & ALM & $\gamma_{\alm} = 0.0022$ & $\mu = 0.5$ \\
    \hline
    \multirow{4}{*}{\shortstack{\textbf{Perceptron} \\ \textbf{regression}}}
        & NC-PDHG & $\begin{aligned}c = 0.55 && \gamma_x = 0.018 && \gamma_y = 0.018\end{aligned}$ & $\alpha = 0.78$\\
        & NC-SPDHG & $\begin{aligned}c = 0.14 && \gamma_x = 0.043 && \gamma_y = 0.0075\end{aligned}$ & $\begin{aligned} \theta = 2m + n = 175 && \alpha = 0.46 \end{aligned}$\\
        & CEG+ & $\gamma_{\ceg} = 0.018$ & $\begin{aligned} \delta_{\ceg} = -0.002 && \alpha_{\ceg} = 0.78 \end{aligned}$\\
        & ALM & $\gamma_{\alm} = 0.0022$ & $\mu = 0.5$ \\
    \hline
    \end{tabular}
    \caption{Chosen parameters for the logistic regression with squared loss and perceptron regression with ReLU activation experiments}
    \label{tab:params}
\end{table}

The first figure, Fig. \ref{fig:linear_convergence}, illustrates the convergence behavior of our algorithms, NC-PDHG and NC-SPDHG, across both experiments. In each case, the sub-figures display linear convergence, a property that could be further explored in future work. The second figure, Fig. \ref{fig:proximal}, compares our algorithms (NC-PDHG and NC-SPDHG) with the state-of-the-art ones (CEG+ and ALM). It shows the number of proximal evaluations required by each algorithm to achieve a specified tolerance in both experiments. Since NC-SPDHG is a randomized algorithm, we include multiple runs in light gray in Fig. \ref{fig:proximal} to highlight the consistency of its performance despite inherent randomness. 
\begin{figure}
    \centering
    \includegraphics[width = \linewidth]{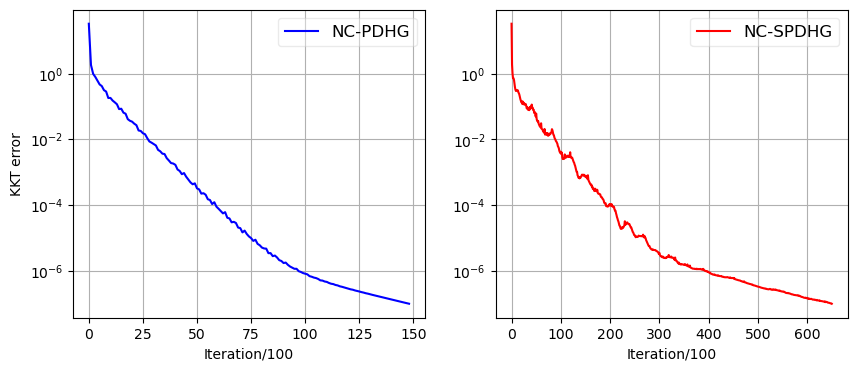} 
    \includegraphics[width=\linewidth]{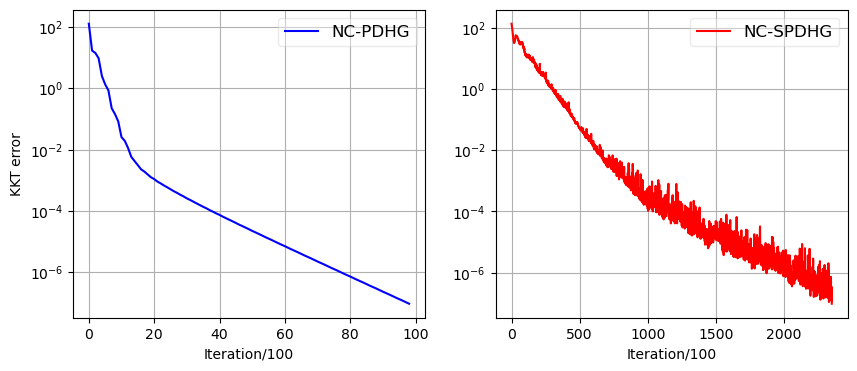}
    \caption{Convergence of NC-PDHG and NC-SPDHG on the logistic regression with squared loss (top) and perceptron regression (bottom) experiments.}
    \label{fig:linear_convergence}
\end{figure}
\begin{figure}
\subfloat[Logistic regression with squared loss]{\includegraphics[width=0.5\linewidth]{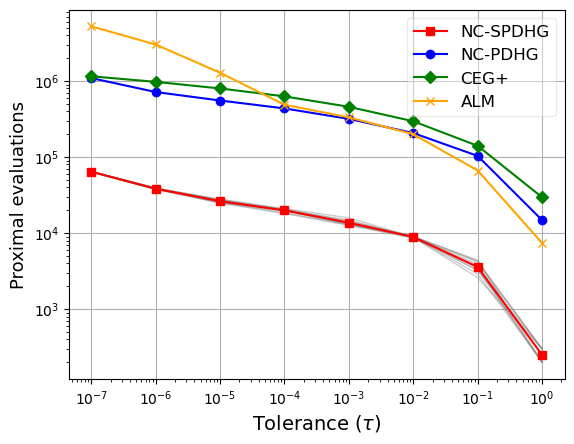}\label{subfig:sigmoid_prox}}
\subfloat[Perceptron regression with ReLU activation]{\includegraphics[width=0.5\linewidth]{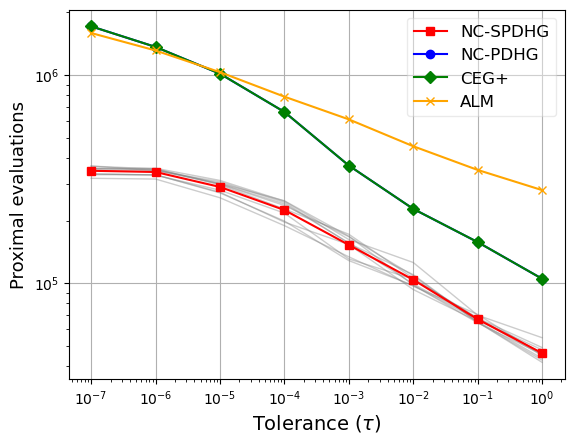}\label{subfig:perceptron_prox}}
\caption{Number of proximal evaluations required to reach different tolerances for the NC-PDHG, NC-SPDHG, CEG+, and ALM algorithms across both experiments. The light gray curves represent multiple runs of the randomized NC-SPDHG algorithm. In Fig. 3b, the curves for NC-PDHG and CEG+ are overlapping}
\label{fig:proximal}
\end{figure}
\subsection{Least-squares} \label{subsec:saga}
The objective of this experiment is to demonstrate that our NC-SPDHG algorithm performs effectively in the convex-concave setting, comparing to leading convex state-of-the-art methods such as SAGA \cite{saga}.

We consider the following \textit{convex} optimization problem:

\begin{equation} \label{eqn:saga_prob}
    \min_{w \in \R^n} \sum_{i  = 1}^m f_i(w) := \sum_{i = 1}^m \frac{1}{2} \left(B_i^\top w - b_i\right)^2 
\end{equation}
where $B_i$ is the $i^{\text{th}}$ row of $B \in \R^{m \times n}$ and $b \in \R^m$ are given. While SAGA addresses problems of the form \eqref{eqn:saga_prob}, NC-SPDHG is designed to solve saddle point problems. Therefore, we reformulate \eqref{eqn:saga_prob} as a \textit{convex-concave} saddle point problem by introducing a new variable $u \in \R^m$ such that $u_i = B_i^\top w, ~\forall i \in \{1, \dots, m\}$. Specifically, we write:
\begin{equation*}
    \begin{split}
        &\min_{w, u} ~ \frac{1}{2}\sum_{i = 1}^m \left(u_i - b_i\right)^2 \\
        &~~\text{s.t.} ~~ u_i = B_i^\top w, ~~ \forall i \in \{1, \dots, m\}
    \end{split}
\end{equation*}
Thus, the associated \textit{convex-concave} saddle point problem writes: 
\begin{equation} \label{eqn:SPP_saga}
    \min_{x \in \R^{n + m}} \max_{y \in \R^m} \frac{1}{2} \sum_{i = 1}^m \left(u_i - b_i \right)^2 + \left\langle Ax, y \right\rangle
\end{equation}
where
\begin{align*}
    x = \begin{pmatrix} w \\ u \end{pmatrix}_{(n+m) \times 1} && A = \begin{bmatrix} B && -\Id_m \end{bmatrix}_{m \times (n+m)}
\end{align*}
By comparing \eqref{eqn:SPP_saga} with the saddle point problem \eqref{eqn:SPP_numerical}, we identify:
\begin{align*}
  g_2 = 0  && f_i(x_i) = \frac{1}{2} (u_i - b_i)^2, ~~\forall i \in \{1, \dots, m\}
\end{align*}

\subsubsection{Problem analysis}
\begin{enumerate}
    \item The SAGA algorithm solves \eqref{eqn:saga_prob} and requires the formula of the gradient of each function $f_i(w) = \frac{1}{2} \left(B_i^\top w - b_i\right)^2$, for all $i \in \{1, \dots, m\}$, which is given by:
\begin{equation*}
    \nabla f_i(w) = B_i \left(B_i^\top w - b_i\right)
\end{equation*}

\item The NC-SPDHG algorithm solves \eqref{eqn:SPP_saga} and requires the proximal operator of each $f_i(x_i) = \frac{1}{2} (u_i - b_i)^2$, for any $i \in \{1, \dots, m\}$, which is given by: 
        \begin{equation*}
            \prox_{\gamma_x f_i}(x_i) = \begin{cases}
                x_i & \text{if} ~i \leq n \\ 
                \frac{x_i + \gamma_x b_i}{1 + \gamma_x}& \text{Otherwise}
            \end{cases}
        \end{equation*}

\item Parameters selection: Since the problem is convex, we guarantee that $\rho \geq 0$, so we choose $\rho = 0$. Then, we choose the NC-SPDHG parameters according to subsection \ref{subsec:parameters_NC_SPDHG} with $c = 0.05$. For SAGA, we choose the step size as $\gamma_{\saga} = \frac{1}{L_{\saga}}$ following \cite[Section 2]{saga}, where $L_{\saga}$ is the largest Lipschitz constant of the gradients of $f_i(w) = \frac{1}{2} \left(B_i^\top w - b_i\right)^2$ among all $i$. That is: $L_{\saga} = \max\left\{\left(\left\|B_i\right\|^2\right)_{1 \leq i \leq m}\right\}$. Table \ref{tab:saga_param} summarizes the chosen parameters.
\begin{table}[H]
    \centering
    \begin{tabular}{||c|c||}
    \hline
    \textbf{NC-SPDHG} & \textbf{SAGA} \\
    \hline
    \hline
        $\begin{aligned}\rho = 0 && \gamma_x = 0.27&& \gamma_y = 0.001 && \alpha = 0.69 \end{aligned}$ & $\gamma_{\saga} = 0.013$\\
        \hline
    \end{tabular}
    \caption{Chosen parameters for the least-squares problem}
    \label{tab:saga_param}
\end{table}
\item KKT error: we consider the KKT error of the unconstrained problem \eqref{eqn:saga_prob}. That is: 
\begin{equation*}
    \K(w) = \left\| \nabla \left(\frac{1}{2} \sum_{i = 1}^m \left(B_i^\top w - b_i\right)^2\right) \right\|^2 = \left\| \nabla \left( \frac{1}{2} \left\| Bw - b \right\|^2\right) \right\|^2 = \left\| B^\top (Bw - b) \right\|^2
\end{equation*}
\end{enumerate}
Fig. \ref{fig:saga} shows that NC-SPDHG performs well and can be a competitive candidate even in the convex-convex settings. 
\begin{figure}[H]
    \centering
    \includegraphics[width=0.7\linewidth]{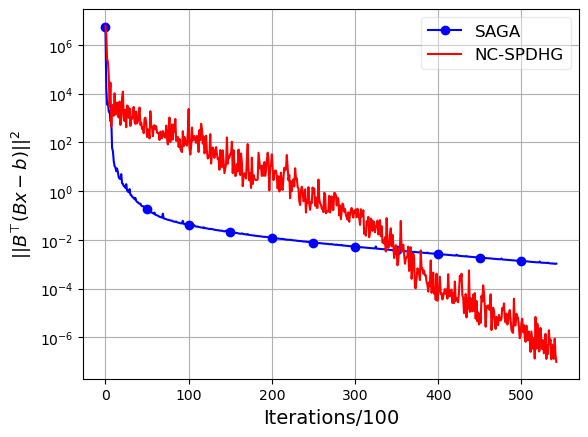}
    \caption{Convergence of SAGA and NC-SPDHG to achieve $\tau = 10^{-7}$-stationary point}
    \label{fig:saga}
\end{figure}
\section{Conclusion and Perspectives} \label{sec:conclusion}

In this paper, we proposed two novel algorithms for solving nonconvex-nonconcave and nonsmooth saddle point problems under the weak MVI assumption. Our algorithms — Algorithm~\ref{alg:NC-PDHG} and Algorithm~\ref{alg:NC-SPDHG}, referred to as NC-PDHG and NC-SPDHG, respectively — extend the classical Primal-Dual Hybrid Gradient (PDHG) and Stochastic Primal-Dual Hybrid Gradient (SPDHG) methods to the nonconvex-nonconcave setting.

Leveraging the notion of the Clarke sub-differential, we established convergence guarantees for both algorithms without requiring any convexity assumptions. In particular, our analysis allows for nonconvex proximals and gradients, assuming only a mild condition on the weak MVI parameter. 

To prove the convergence of NC-SPDHG, we relied on the \texttt{PEPit} performance estimation toolbox and auto-Lyapunov techniques, which were essential in enabling the analysis. For both algorithms, we discussed parameter selection and validated our findings through a range of numerical experiments that illustrate linear convergence of both algorithms. 

\vspace{0.3cm}
\noindent
This work opens several promising directions for future research:
\begin{itemize}
\item \textbf{Linear convergence.} Proving under which conditions the numerically observed linear convergence of both algorithms holds.

\item \textbf{Generalization to nonzero $f_2$.} Extending NC-SPDHG to handle the general saddle point problem \eqref{eqn:SPP} with $f_2 \neq 0$. 

\item \textbf{Beyond separability.} Like many coordinate-based algorithms, our methods assume separability of the function where coordinate updates are applied. It would be interesting to investigate whether this assumption can be relaxed.

\item \textbf{Non-bilinear couplings.} Our analysis assumes bilinear coupling in the smooth term $ \varphi(x, y) = f_2(x) + \langle Ax, y \rangle - g_2(y) $. A natural extension would be to study more general forms of $ \varphi(x, y) $ that include nonlinear interactions between $x$ and $y$.

\item \textbf{PURE-CD extension.} In (NC)-SPDHG, all the coordinates of the dual variable $y$ are updated at each iteration. It would be interesting to update only some coordinates of $y$ following the sparsity pattern of $A$, as is done in PURE-CD~\cite{PURE_CD} for convex-concave saddle point problems. 

\end{itemize}
\appendix
\section{Projection onto the graph of ReLU} \label{apdx:proj}
In this appendix, we detail the computation of the projection onto the graph of the ReLU function, as used in the perceptron regression experiment. Specifically, we wish to compute:
$$
(\tilde u, \tilde l) \in \proj_{\gp(r)}(u, l), \quad \text{where} \quad r(u) := \max(0, u).
$$
Let's analyze the optimality conditions  based on the sub-differential:

\begin{align*}
    (\tilde u,& \;r(\tilde u)) \in \proj_{\gp(r)}(u, l)
    = \arg\min_{(u' l') \in \gp(r)} \|(u', l') - (u, l)\|^2 \\ 
&\Longrightarrow   \tilde u \in \arg\min_{u' \in \R} \|(u' - u, \max(0, u') - l)\|^2 
    = \arg\min_{u' \in \R} (u' - u)^2 + (\max(0, u') - l)^2 \\
  &  \Longrightarrow ~  0 \in \partial \left((. - u)^2 + (\max(0, .) - l)^2 \right)(\tilde u) \\ 
   & \Longrightarrow ~  0 \in 2(\tilde u - u) + 2(\max(0, \tilde u) - l) \times \begin{cases}
        0, & \tilde u < 0 \\
        1, & \tilde u > 0 \\
        [0, 1], & \tilde u = 0
    \end{cases}
\end{align*}

So, we analyze three main cases:
\begin{enumerate}
    \item $\tilde{u} < 0 \Longrightarrow \max(0, \tilde{u}) = 0 \Longrightarrow \tilde{u} = u$
    \item $\tilde{u} > 0 \Longrightarrow \max(0, \tilde{u}) = \tilde u \Longrightarrow \tilde{u} = \frac{u + l}{2} $
    \item $ \tilde{u} = 0 \Longrightarrow \max(0, \tilde u) \in [0, 1] \Longrightarrow u \in -l [0, 1]$
\end{enumerate}
Now, as the solution is not unique, we are going to do case by case to find the correct stationarity point. 
\subsection*{Case analysis in short}
\renewcommand{\arraystretch}{1.4}
\begin{tabular}{|c|c|c|c|}
\hline
Case & Condition on \( u, l \) & Stationary \( \tilde{u} \) & Result \\
\hline
\hline
1 & \( u \geq 0, u + l \leq 0 \) & \( \tilde{u} = 0 \) & \( \tilde{u} = 0 \) \\
2 & \( u < 0, u + l < 0 \) & \( \tilde{u} = u \) & \( \tilde{u} = u \) \\
3 & \( u < 0, u + l = 0 \) & Compare \( u \) vs \( 0 \) & \( \tilde{u} = u \) \\
4 & \( u > 0, u + l > 0 \) & \( \tilde{u} = \frac{u + l}{2} \) & \( \tilde{u} = \frac{u + l}{2} \) \\
5 & \( u = 0, u + l > 0 \) & Compare \( 0 \) vs \( \frac{l}{2} \) & \( \tilde{u} = \frac{l}{2} \) \\
6 & \( u < 0, u + l > 0 \) & Compare all 3 candidates & Decision based on \( (1 + \sqrt{2})u + l \) \\
\hline
\end{tabular}

\medskip

\noindent In particular, Case 6 requires comparing distances. After analysis, the selection criterion is:
$$
\tilde{u} =
\begin{cases}
u, & \text{if } (1 + \sqrt{2}) u + l \leq 0 \\
\frac{u + l}{2}, & \text{otherwise}
\end{cases}
\quad \text{for } u < 0, u + l > 0.
$$

\subsection*{Case analysis in details}
\begin{enumerate}
    \item[Case 1:] $u \geq 0$ and $u + l \leq 0$
    \begin{itemize}
        \item[Case 1.1:] If $\tilde u < 0$, then $\tilde u  = u \geq 0$. Not possible.
        \item[Case 1.2:] If $\tilde u > 0$, then $\tilde u = \frac{u + l}{2} \leq 0$. Not possible 
        \item[Case 1.3:] If $\tilde u = 0$, then $u \in -l [0, 1]$. Since $u \geq 0$ and $u + l \leq 0$, then $l \leq 0$. So $u \in [0, -l] \iff 0 \leq u \leq -l \iff l \leq u + l \leq 0$. Possible
    \end{itemize}
    Thus, $\tilde u = 0$
    \vspace{0.15cm}
    \item[Case 2:] $u < 0 $ and $u + l < 0$
    \begin{itemize}
        \item[Case 2.1:] If  $\tilde u < 0$, then $\tilde u = u < 0$. Possible. 
        \item[Case 2.2:] If $\tilde u > 0$, then $\tilde u = \frac{u + l}{2} < 0$. Not possible. 
        \item[Case 2.3:] If $\tilde u = 0$, then $u + l \in -l[-1, 0]$. If $l \geq 0$, then $u + l \in [0, l]$ which is not possible, and if $l < 0$, then $u + l \in [l, 0]$ which means that $u + l \geq l$ but $u < 0$, so it's not possible. 
    \end{itemize}
    Thus, $\tilde u  = u$
        \vspace{0.15cm}
    \item[Case 3:] $u < 0$ and $u + l = 0$
    \begin{itemize}
        \item[Case 3.1:] If $\tilde u < 0$, then $\tilde u = u < 0$. Possible
        \item[Case 3.2:] If $\tilde u > 0$, then $\tilde u = \frac{u + l}{2} = 0$. Not possible
        \item[Case 3.3:] If $\tilde u = 0$, then $u \in -l[0, 1]$. Possible
    \end{itemize}
    There are 2 possible candidates, we have to compare between them: 
    \begin{align*}
        \dist \left((u, 0), (u, l)\right)^2 = l^2 && \dist\left((0, 0), (u, l)\right)^2 = u^2 + l^2 > l^2
    \end{align*}
    Thus, $\tilde u = u$
        \vspace{0.15cm}
    \item[Case 4:] $u > 0$ and $ u + l > 0$
    \begin{itemize}
        \item[Case 4.1:] $\tilde u < 0$, then $\tilde u = u > 0$. Not possible 
        \item[Case 4.2:] If $\tilde u > 0$, then $\tilde u = \frac{u + l}{2} > 0$. Possible
        \item[Case 4.3:] If $\tilde u = 0$, then $u \in -l[-1, 0]$. If $l \geq 0$, then $u + l \in [0, l]$ which means $u + l \leq l$ but $u > 0$ so it's not possible. If $l < 0$, then $u + l \in [l, 0]$ so $u + l \leq 0$ which is not possible as well. 
    \end{itemize}
    Thus, $\tilde u = \frac{u + l}{2}$
        \vspace{0.15cm}
    \item[Case 5:] $u  = 0$ and $u + l > 0$
    \begin{itemize}
        \item[Case 5.1:] If $\tilde u < 0$, then $\tilde u = u = 0 $. Not possible
        \item[Case 5.2:] If $\tilde u > 0$, then $ \tilde u = \frac{u + l}{2} > 0$. Possible
        \item[Case 5.3:] If $\tilde u = 0$, then $u \in -l[0, 1]$. If $l \geq 0$, then $u + l \in [0, l]$ which is possible. If $l < 0$, then $u + l \in [l, 0]$ which means that $u + l \leq 0$, so it's not possible.
    \end{itemize}
    There are 2 possible candidates to compare between: 
    \begin{align*}
        \dist\left(\left(\frac{u+l}{2}, \frac{u+l}{2}\right), (u, l)\right)^2 = \frac{l^2}{2} && \dist\left((0, 0), (u, l)\right)^2 = l^2 > \frac{l^2}{2}
    \end{align*}
    Thus, $\tilde u = \frac{u+l}{2}$
        \vspace{0.15cm}
    \item[Case 6:] $u < 0$ and $u + l > 0$
    \begin{itemize}
        \item[Case 6.1:] If $\tilde u < 0$, then $\tilde u = u < 0$. Possible
        \item[Case 6.2:] If $\tilde u > 0$, then $\tilde u = \frac{u + l}{2} > 0$. Possible
        \item[Case 6.3:] If $\tilde u = 0$, then $u \in -l[0, 1]$. If $l > 0$, then $u + l \in [0, l]$ which means $0 \leq u + l \leq l$, possible. If $l \leq 0$, then $ u + l \in [l, 0]$ which is not possible since $u + l > 0$.
    \end{itemize}
    There are 3 possible candidates: $\tilde u \in \{u, \frac{u+l}{2}, 0\}$. So, let's compare the squared distances: 
    \begin{align*}
        \dist\left((u, 0), (u, l)\right)^2 &= l^2 \\
        \dist\left(\left(\frac{u+l}{2}, \frac{u+l}{2}\right), (u, l)\right)^2 &= \frac{(u- l)^2}{2} \\
        \dist\left((0, 0), (u, l)\right)^2 &= u^2 + l^2
    \end{align*}
    Clearly, $\tilde u = 0$ is never a solution. However, to know whether $\tilde u = u$ or $\tilde u = \frac{u+l}{2}$, we have to know when $l^2 \leq \frac{(u - l)^2}{2}$. Note that:
    $$l^2 \leq \frac{(u - l)^2}{2} \iff \frac{1}{2} l^2 + ul - \frac{1}{2}u^2 \leq 0 $$
    Fixing $u$, the discriminant is: $\Delta = 2u^2 > 0$. Thus, 
    $$(\sqrt{2} - 1) u \leq l \leq -(1 + \sqrt{2}) u$$
    The LHS of the inequality is useless since we already assume that $u + l > 0$. The RHS provides us with the projection criterion: If $l \leq -(1 + \sqrt{2}) u$, then $\tilde u = u$, otherwise $\tilde u = \frac{u + l}{2}$.
\end{enumerate}

\subsection*{Final Expression}
Combining all cases, the closed-form projection is:
\[
\tilde{u} = \begin{cases}
0, & u \geq 0, u + l \leq 0 \\
\frac{u + l}{2}, & u \geq 0, u + l > 0 \\
\frac{u + l}{2}, & u < 0, (1 + \sqrt{2}) u + l > 0 \\
u, & u \leq 0, (1 + \sqrt{2}) u + l \leq 0 
\end{cases}
\]

Then, \( \tilde{l} = r(\tilde{u}) = \max(0, \tilde{u}) \).

\begin{acknowledgements}
This work was supported by the Agence National de la Recherche grant ANR-20-CE40-0027, Optimal Primal-Dual Algorithms (APDO), and T\'el\'ecom Paris's research and teaching chair Data Science and Artificial Intelligence for Digitalized Industry and Services DSAIDIS.
\end{acknowledgements}


\paragraph{\emph{{\small\textbf{Conflict of interest}}}}
The authors declare that they have no conflict of interest.

\bibliographystyle{spmpsci}      
\bibliography{ref}   


\end{document}